\documentclass[
    pdflatex,
    sn-mathphys-ay
]{sn-jnl}

\usepackage{graphicx}
\usepackage{amssymb}
\usepackage{amsthm}
\usepackage[rgb]{xcolor}  

\usepackage[american]{babel}  
\usepackage[shortlabels]{enumitem}  
\usepackage{tikz}  
\usepackage{calc}  
\usepackage{pgfplots}  
\usepackage{mathtools}  
\usepackage{physics}  
\usepackage{icomma}  
\usepackage{siunitx}  
\usepackage{csquotes}  
\usepackage[
    figurename = {Fig.},
    labelsep = {space}
]{caption}  
\usepackage[
    singlelinecheck = true
]{subcaption}  
\usepackage{cleveref}  

\newlist{properties}{enumerate}{1}
\setlist{nosep}
\setlist[enumerate]{label = (\roman*), ref = \roman*}
\setlist[properties]{align = left, labelindent = 0.5em, leftmargin = *}

\theoremstyle{thmstylethree}

\newtheorem{theorem}{Theorem}[section]

\newtheorem{corollary}{Corollary}[section]
\newtheorem{example}{Example}[section]
\newtheorem{remark}{Remark}[section]

\numberwithin{equation}{section}
\numberwithin{figure}{section}
\numberwithin{table}{section}

\newcommand*{\R}{\mathbb{R}}
\newcommand*{\N}{\mathbb{N}}
\newcommand*{\Lcal}{\mathcal{L}}
\newcommand*{\Pcal}{\mathcal{P}}

\newcommand*{\comma}{\text{\,,}}
\newcommand*{\period}{\text{\,.}}
\renewcommand*{\colon}{:\;\;}

\newcommand*{\rdb}[1]{\left(#1\right)}
\newcommand*{\sqb}[1]{\left[#1\right]}
\newcommand*{\clb}[1]{\left\{#1\right\}}

\newcommand*{\transposed}{\text{T}}

\DeclareMathOperator{\cone}{cone}
\DeclareMathOperator{\conv}{conv}
\DeclareMathOperator{\rk}{rk}
\DeclareMathOperator{\extdir}{ext\,dir}
\DeclareMathOperator{\Eff}{Eff}
\DeclareMathOperator{\gEff}{\overline{Eff}}
\DeclareMathOperator{\Min}{Min}
\DeclareMathOperator{\gMin}{\overline{Min}}
\newcommand*{\conemin}[1]{\textstyle \min_{#1}}

\newcommand*{\set}[2]{\left\{ #1 \;\middle|\; #2 \right\}}
\newcommand*{\inlineset}[2]{\{ #1 \;\vert\; #2 \}}

\newcommand*{\recc}{0^+}

\crefname{section}{Section}{Sections}
\crefname{subsection}{Section}{Sections}
\crefname{equation}{Equation}{Equations}
\crefname{figure}{Figure}{Figures}
\crefname{table}{Table}{Tables}

\crefname{definition}{Definition}{Definitions}
\crefname{theorem}{Theorem}{Theorems}
\crefname{proposition}{Proposition}{Propositions}
\crefname{lemma}{Lemma}{Lemmas}
\crefname{corollary}{Corollaray}{Corollaries}
\crefname{example}{Example}{Examples}
\crefname{remark}{Remark}{Remarks}

\crefname{enumi}{part}{parts}
\crefname{propertiesi}{property}{properties}
\creflabelformat{enumi}{(#2#1#3)}

\definecolor{setcolor}{rgb}{0.9,0.9,0.9}  
\definecolor{plotcolor1}{rgb}{1,0,0.2}  
\definecolor{plotcolor2}{rgb}{0,0.2,1}  

\usetikzlibrary{
    arrows.meta,
    calc,
    decorations.markings,
    patterns.meta,
    plotmarks,
    datavisualization.formats.functions
}
\tikzset{
    font = \normalsize,
    > = {Latex[length = 6pt]},
    ultra thin/.style = {line width = 0.3pt},
    very thin/.style = {line width = 0.35pt},
    thin/.style = {line width = 0.4pt},
    semithick/.style = {line width = 0.5pt},
    thick/.style = {line width = 0.75pt},
    very thick/.style = {line width = 1pt},
    ultra thick/.style = {line width = 1.4pt},
    every path/.append style = {thin},
    point/.style = {fill, radius = 1.8pt},
    point scaled/.style = {fill, radius = 1.2pt},
    point scaled2/.style = {fill, radius = 1.5pt},
    set/.style = {fill = setcolor},
    ordering cone line/.style = {},
    ordering cone fill/.style = {
        pattern = {Lines[angle = -30, distance = 4pt]},
        pattern color = black
    },
    ordering cone/.style = {
        ordering cone line,
        ordering cone fill
    },
    minimal points/.style = {ultra thick},
    objective/.style = {semithick},
    upper image/.style = {dashed},
    every mark/.append style = {
        mark size = 3pt,
        line width = 0.75pt
    }
}

\pgfplotsset{compat = 1.18}
\usepgfplotslibrary{groupplots,statistics}
\pgfplotsset{
    every axis/.append style = {
        enlarge x limits = 0.05,
        enlarge y limits = 0.1,
        solid,
        only marks,
        cycle list name = mylist
    },
    every axis y label/.append style = {
        at = {($(ticklabel cs:0.5) + (0.5ex, 0ex)$)}
    },
    every axis x label/.append style = {
        at = {($(ticklabel cs:0.5) + (0ex, 0.5ex)$)}
    },
    every tick/.style = {
        black},
    full width/.style = {
        width = 0.6\linewidth
    },
    overfull width/.style = {
        width = 0.62\linewidth
    },
    basic heigth/.style = {
        height = 0.26\textheight
    },
    basic/.style = {
        basic heigth,
        full width,
        minor tick num = 1
    },
    basic q/.style = {
        basic,
        xtick distance = 1,
        minor tick num = 0
    },
    every boxplot/.append style = {
        black,
        mark = |,
        every mark/.append style = {line width = 0.25pt}
    },
    boxplots/.style = {
        overfull width,
        basic heigth,
        boxplot/draw direction = x,
        boxplot/box extend = 0.75,
        ytick distance = 1,
        minor tick num = 0
    },
    every axis plot/.append style = {
        thick
    },
    every axis legend/.append style = {
        legend cell align = left,
        font = \normalsize
    },
    legend top left/.style = {
        every axis legend/.append style = {
            at = {(0, 1)},
            anchor = north west,
            xshift = {0.2cm},
            yshift = {-0.2cm}
        }
    },
    x-q/.style = {
        xlabel = {number $q$ of objectives}
    },
    x-ratio/.style = {
        xlabel = {computing time ratio}
    },
    y-q/.style = {
        ylabel = {number $q$ of objectives}
    },
    y-k/.style = {
        ylabel = {rank $k$ of objective matrix}
    },
    y-time/.style = {
        ylabel = {computing time in \si{\s}}
    }
}

\makeatletter
\pgfplotsset{
    groupplot xlabel/.initial={},
    every groupplot x label/.style={
        at={($({group c1r\pgfplots@group@rows.west}|-{group c1r\pgfplots@group@rows.outer south})!0.5!({group c\pgfplots@group@columns r\pgfplots@group@rows.east}|-{group c\pgfplots@group@columns r\pgfplots@group@rows.outer south})$)},
        anchor=north,
    },
    groupplot ylabel/.initial={},
    every groupplot y label/.style={
        rotate=90,
        at={($({group c1r1.north}-|{group c1r1.outer
        west})!0.5!({group c1r\pgfplots@group@rows.south}-|{group c1r\pgfplots@group@rows.outer west})$)},
        anchor=south
    },
    execute at end groupplot/.code={%
    \node [/pgfplots/every groupplot x label]
    {\pgfkeysvalueof{/pgfplots/groupplot xlabel}};
    \node [/pgfplots/every groupplot y label]
    {\pgfkeysvalueof{/pgfplots/groupplot ylabel}};
    },
    group/only outer labels/.style =
        {
        group/every plot/.code = {%
            \ifnum
                \pgfplots@group@current@row=\pgfplots@group@rows\else%
                \pgfkeys{xticklabels = {}, xlabel = {}}
            \fi%
            \ifnum
                \pgfplots@group@current@column=1\else%
                \pgfkeys{yticklabels = {}, ylabel = {}}
            \fi%
        }
    }
}
\def\endpgfplots@environment@groupplot{%
    \endpgfplots@environment@opt%
    \pgfkeys{/pgfplots/execute at end groupplot}%
    \endgroup%
}
\makeatother

\pgfplotscreateplotcyclelist{mylist}{
        {plotcolor1, mark = x},
        {plotcolor2, mark = +},
        {olive, mark = Mercedes star},
        {orange, mark = Mercedes star flipped},
        {teal, mark = star},
        {violet, mark = asterisk}
}

\begin{document}


    \title{Low-Rank Multi-Objective Linear Programming}

    \author[1]{\fnm{Andreas} \sur{Löhne}}
    \author[2]{\fnm{Pascal} \sur{Zillmann}}

    \affil[1]{\orgname{Friedrich Schiller University}, \orgdiv{Department of Mathematics}, \orgaddress{\postcode{07737} \city{Jena}, \country{Germany}}, andreas.loehne@uni-jena.de, ORCID: 0000-0003-0872-4735}
    \affil[2]{\orgname{Friedrich Schiller University}, \orgdiv{Department of Mathematics}, \orgaddress{\postcode{07737} \city{Jena}, \country{Germany}}, pascal.zillmann@uni-jena.de, ORCID: 0009-0005-5682-8506, Corresponding Author}

    \abstract{%
        When solving multi-objective programs (MOLPs), the number of objectives essentially determines the computing time.
        This can even lead to practically unsolvable problems.
        Consequently, it is of huge interest to reduce the number of objectives without losing information.
        There exist approaches for transforming vector linear programs (VLPs), i.e., the ordering cone is non-standard, into MOLPs in the literature.
        We here propose a method for solving MOLPs more efficiently by applying the transformation in reverse.
        In particular, we discuss MOLPs with linear dependent objective functions.
        The resulting VLPs have merely as many objectives as the rank of the objective matrices of the MOLPs.
        To achieve this, only a factorization of this matrix needs to be calculated.
        One factor then forms a new ordering cone, while the other remains as the objective matrix.
        Through multiple series of numerical experiments, we show that this approach indeed significantly reduces computing time, and therefore provides a technique for solving low-rank MOLPs in practice.
        As there are fewer objectives to consider, the approach additionally helps decision makers to get a better visualization as well as understanding of the actual problem.
        Moreover, we will point out that the equivalence between MOLPs and corresponding VLPs can be used to derive statements about the well-known concept of nonessential objectives.%
    }



    \maketitle



    \section{Introduction}
    \label{sec:introduction}

    There are lots of solution algorithms for multi-objective linear programs (MOLPs) that operate in the objective space.
    These algorithms aim to find minimal points of the image of the MOLP.
    Among these methods are simplex-like methods, see, for example,~\citet{Dauer_Saleh_1990,Dauer_Liu_1990,Dauer_1993,Rudloff_etal_2017}.
    There are also cutting plane methods such as Benson’s algorithm~\citep{Benson_1998,Benson_1998a,Benson_1998b}, as well as its variants and extensions~\citep{Loehne_2011,Hamel_etal_2014,Csirmaz_2016}, dual variants~\citep{Ehrgott_etal_2012,Hamel_etal_2014}, approximative algorithms~\citep{Shao_Ehrgott_2008,Shao_Ehrgott_2008a}, and the NISE algorithm~\citep{Cohon_etal_1979,Cohon_1978}.
    In addition, the dominance usually becomes weaker in higher dimensions since the fraction of the ambient space occupied by the ordering cone decreases.
    Therefore, computing time increases rapidly with the number of objectives and quickly exceeds what can be calculated.
    However, optimization problems with many objectives are common in practice (whether linear or nonlinear).
    The list in~\citet{CoelloCoello_etal_2002} includes several problems with ten or more objectives.
    Moreover, \citet{Csirmaz_2016} employs multi-objective linear optimization to map an entropy region of jointly distributed random variables.
    One example of a nonlinear program is deciding on the most environmentally friendly design for photovoltaic systems~\citep{Perez-Gallardo_etal_2018}.
    Additionally, multi-objective problems with many objectives are considered in evolutionary multi-objective optimization, e.g., \citet{CoelloCoello_etal_2002,Rosenthal_2023}.

    Due to the large number of objectives, it is valuable to reduce the number of objectives without losing information if possible.
    Many approaches have been investigated to reduce the dimension of the objective space for nonlinear programs.
    For instance, \citet{Brockhoff_Zitzler_2006,Brockhoff_Zitzler_2006a,Brockhoff_Zitzler_2009,Saxena_etal_2013,Pozo_etal_2012} examined which objectives are essential for describing the set of minimal points and which can be eliminated.
    \citet{LopezJaimes_etal_2014,Cheung_etal_2016,Lindroth_etal_2010,Perez-Gallardo_etal_2018} simplified the problem by constructing a smaller set of objectives from the actual problem.

    Moreover, \citet{Gal_Leberling_1977,Gal_1980,Malinowska_2006,Thoai_2012} studied the concept of nonessential objectives in the context of multi-objective linear programming.
    An important result is that objectives can be deleted if they can be represented as a conic combination of the other objectives.
    Thus, objectives can be ignored when they are linearly dependent in a particular manner.
    However, not all linear combinations are conic combinations, and therefore we ask whether it is possible to eliminate linear dependence completely.

    Given a vector linear program (VLP) whose ordering cone is polyhedral and therefore determined by a matrix, we can multiply the objective matrix by this matrix and replace the ordering cone with the standard ordering cone.
    We then obtain an equivalent MOLP.
    This well-known method is presented in~\citet[Lemma~2.3.4]{Sawaragi_etal_1985}; \citet[Theorem~7.5]{Yu_1985}; \citet[Proposition~4.1]{Weidner_1990}; \citet[Theorem~4.1]{Engau_Wiecek_2007}.
    Also, \citet[Lemma~1.18]{Eichfelder_2008} adresses such a transformation approach in a more general setting.
    Motivated by the question of what happens when objectives are linearly combined or deleted, \citet[Theorem~3.1]{Dempe_etal_2015} apply the method to arbitrary vector optimization problems with arbitrary objective functions, nonempty closed feasible sets, and convex ordering cones.
    Moreover, \citet[Theorem~3.2]{Cambini_etal_2003} discuss the method with respect to several kinds of minimality, and \citet{Engau_Wiecek_2007} consider relations for approximate minimality.

    In the literature, optimization problems are always transformed to reduce them to simpler problems.
    For instance,~\citet[Remark~7.12]{Yu_1985}; \citet[page~17]{Eichfelder_2008} describe how ordering relations induced by arbitrary polyhedral ordering cones may be replaced by the natural order.
    However, we can also do this in the converse direction: Starting with a factorization of the objective matrix of an MOLP yields an equivalent VLP by defining the ordering cone through the one factor and using the other one as the objective matrix.
    If the objective matrix does not have a full row rank, a suitable factorization results in a problem with merely as many objectives as the rank of the original objective matrix.
    This approach eliminates linear dependence but introduces a new partial order.
    Moreover, with this approach, we can simplify and generalize the results about nonessential objectives in~\citet{Gal_Leberling_1977,Gal_1980,Malinowska_2006,Thoai_2012}.

    In practice, data is often subject to uncertainty.
    Therefore, the objective matrix can be approximated by a lower-rank matrix without any loss of information content within the range of uncertainty.
    For instance, if the smallest singular values of the objective matrix are below a certain level of uncertainty, they do not properly contribute to the problem.
    One could then delete these singular values from the singular value decomposition of the objective matrix.
    If the $k$~largest singular values are retained, the resulting approximation has rank~$k$.
    Due to the Eckart-Young theorem~\citep{Eckart_Young_1936} and its generalization by~\citet{Mirsky_1960}, this approximation is optimal w.r.t.\ unitarily invariant matrix norms, such as the Frobenius norm and the Euclidean operator norm~\citep[Theorem~2.4.8]{Golub_VanLoan_2013}, over all matrices of rank at most~$k$.
    Decision makers may find it sufficent to consider the simplified problem, which offers the advantages of a lower solution time and greater interpretability.

    The concept of low-rank approximation is closely related to principal component analysis (PCA), cf.~\citet{Pearson_1901,Jolliffe_1986}.
    Given a centered data matrix, PCA seeks a small number of orthogonal directions that capture most of the variance in the data.
    Mathematically, this is equivalent to finding a rank-$k$ approximation of the matrix that minimizes the Frobenius norm of the error.
    In this sense, replacing the objective matrix in a multiobjective optimization problem by a low-rank approximation can be interpreted as projecting the objectives onto the \enquote{most informative} directions, similar to PCA.

    This article is structured as follows.
    In \cref{sec:problem_setting_preliminaries}, we introduce the notions of an MOLP and a VLP in~$\R^q$ as well as the concepts of minimality and generalized minimality.
    \cref{sec:reduction_of_objectives} presents the aforementioned transformation approach.
    We prove that both the MOLP and the VLP are indeed equivalent problems and explain how the VLP can be interpreted.
    Additionally, we examine the case when the decomposition is not of minimal dimensions, and state a more advanced theorem on the equivalence of MOLP and VLP in a generalized sense.
    \cref{sec:nonessential_objectives} focuses on the concept of nonessential objectives.
    We prove the statement on conic combinations of objectives using the reformulation of the MOLP into an equivalent VLP with the power of the concept of generalized minimality from \cref{sec:reduction_of_objectives}.
    We relate the elimination of nonessential objectives to facet or vertex enumeration of the ordering cone.
    In \cref{sec:examples}, we demonstrate the effectiveness of the transformation approach through a series of numerical experiments and therefore show that it provides a powerful technique for solving MOLPs in practice.


    \section{Problem Setting and Preliminaries}
    \label{sec:problem_setting_preliminaries}

    In this section, we introduce multi-objective linear programs (MOLP) and---as a generalization---vector linear programs (VLP) in $\R^q$.
    To this end, we first need to clarify the notion of minimality.

    Consider the preorder~$\leq_C$ on $\R^q$ generated by a convex cone~$C \subseteq \R^q$ \citep{Goepfert_Nehse_1990,Jahn_2011,Loehne_2011}, i.e.,
    \begin{equation}
        \label{eq:preorder}
        \leq_C \coloneq \set{(z^1, z^2) \in \R^q \times \R^q}{z^2 - z^1 \in C} \!\period
    \end{equation}
    We will always write $z^1 \leq_C z^2$ instead of $(z^1, z^2) \in \ \leq_C$.
    We omit the index if $C$ equals the cone~$\R_+^q \coloneq \inlineset{z \in \R^q}{\forall\, j = 1, \ldots, q : z_j \geq 0}$; we then call $\leq$ the \emph{natural order} on $\R^q$.
    When we say $z^1 \lneq_C z^2$, we mean $z^1 \leq_C z^2$ and $z^1 \neq z^2$.
    Note that $\leq_C$ is a partial order if and only if $C$ is pointed, i.e., $C \cap (-C) = \clb{0}$.

    Given a nonempty subset $A \subseteq \R^q$, a point~$\bar{z} \in A$ is called \emph{minimal} w.r.t.\ $\leq_C$ (or \emph{$C$-minimal}) in~$A$ if
    \begin{equation}
        \label{eq:definition_minimal_point}
        \nexists\, z \in A \colon z \lneq_C \bar{z}
    \end{equation}
    or, equivalently,
    \begin{equation}
        \label{eq:definition_minimal_point_alternativ}
        \forall\, z \in A \colon
        z \leq_C \bar{z} \Longrightarrow z = \bar{z}
        \period
    \end{equation}
    The set of all minimal points of~$A$ is denoted by $\Min_C(A)$.
    Note that a point~$\bar{z} \in A$ is a minimal point of~$A$ if and only if there is no distinct point~$z \in A$ being an element of $\bar{z} - C$ (with $\bar{z} - C$ denoting the Minkowski sum of $\clb{\bar{z}}$ and $-C$).
    Therefore, the set of minimal points can be written as
    \begin{align}
        \notag
        \Min_C(A)
        &= \set{\bar{z} \in A}{A \cap (\bar{z} - C) = \clb{\bar{z}}} \\
        \label{eq:minimal_points_set}
        &= \set{\bar{z} \in A}{A \cap (\bar{z} - C \setminus \clb{0}) = \emptyset}
        \!\period
    \end{align}
    In \cref{fig:minimal_points}, some examples are given.

    \begin{figure}[hbt]
        \centering
        \begin{subfigure}{0.49\linewidth}
            \centering
            \begin{tikzpicture}[scale = 1.2]
    \node[below left] at (0, 0) {0};
    \draw[->] (-0.4, 0) -- (3.5, 0) node[below] {$z_1$};
    \draw[->] (0, -0.4) -- (0, 3) node[left] {$z_2$};

    \draw[set, shift = {(1, 0.5)}] (1, 0) -- (2, 0) -- (2, 2) -- (0, 2) -- (0, 1.5) -- (0.25, 0.5) -- cycle;
    \node[shift = {(1, 0.5)}] at (1.2, 1.2) {$A$};

    \draw[minimal points, shift = {(1, 0.5)}] (0, 1.5) -- (0.25, 0.5) -- (1, 0);

    \draw[ordering cone, shift = {(1, 0.5)}] (0.2, 0.7)+(-1, 0) -- (0.2, 0.7) -- +(0, -1) node[right]{$\bar{z} - \R^q_+$};

    \draw[point scaled2, shift = {(1, 0.5)}] (0.2, 0.7) node[above right = -1pt]{$\bar{z}$} circle;

\end{tikzpicture}
            \caption{}
            \label{subfig:minimal_points_1}
        \end{subfigure}%
        \begin{subfigure}{0.49\linewidth}
            \centering
            \begin{tikzpicture}[scale = 1.2]
    \node[below left] at (0, 0) {0};
    \draw[->] (-0.4, 0) -- (3.5, 0) node[below] {$z_1$};
    \draw[->] (0, -0.4) -- (0, 3) node[left] {$z_2$};

    \draw[set, shift = {(1, 0.5)}] (1, 0) -- (2, 0) -- (2, 2) -- (0, 2) -- (0, 1.5) -- (0.25, 0.5) -- cycle;
    \node[shift = {(1, 0.5)}] at (1.2, 1.2) {$A$};

    \draw[minimal points, shift = {(1, 0.5)}] (0, 2) -- (0, 1.5) -- (0.25, 0.5) -- (1, 0) -- (2, 0);

    \draw[ordering cone, shift = {(1, 0.5)}] (0.2, 0.7)+(-170:1) -- (0.2, 0.7) -- +(-115:1) node[right]{$\bar{z} - C$};

    \draw[point scaled2, shift = {(1, 0.5)}] (0.2, 0.7) node[above right = -1pt]{$\bar{z}$} circle;

\end{tikzpicture}
            \caption{}
            \label{subfig:minimal_points_2}
        \end{subfigure}
        \caption{
            Examples for a minimal point~$\bar{z}$ of a subset~$A \subseteq \R^2$ w.r.t.\ (a)~the natural order and (b)~a partial order generated by another cone~$C$.
            $\Min_C(A)$ is highlighted by bold lines.
            The (translated negative) ordering cones are shaded.
        }
        \label{fig:minimal_points}
    \end{figure}

    Now, let $C$ be a polyhedral convex cone, i.e., $C = \inlineset{z \in \R^q}{B z \geq 0}$ for some matrix $B \in \R^{r \times q}$.
    Furthermore, let $S \subseteq \R^n$ be a polyhedron and $P \in \R^{q \times n}$ a matrix.
    Then, we call
    \hypertarget{P}{
        \begin{flalign}
            \label{eq:vlp}
            &\text{(P)}
            &\conemin{C} \quad P x \qq{s.t.} x \in S
            &&
        \end{flalign}
    }%
    \newcommand*{\Pref}{(\hyperlink{P}{P})}%
    \newcommand*{\EffPref}{\Eff(\hyperlink{P}{\text{P}})}%
    \newcommand*{\gEffPref}{\gEff(\hyperlink{P}{\text{P}})}%
    a \emph{vector linear program} (VLP).
    $S$ is referred to as the \emph{feasible set}, its image~$P[S]$ under~$P$ is named the \emph{image} of~\Pref, and~$P[S] + C$ is called the \emph{upper image} of~\Pref.
    \emph{Solving} the VLP is the task to find feasible points~$\bar{x} \in S$ whose images under~$P$ are $C$-minimal in the image of~\Pref, i.e., $P \bar{x} \in \Min_C P[S]$.
    These points are called \emph{efficient} and form the set
    \begin{align}
        \notag
        \EffPref
        \coloneq&\; P^{-1}[\Min_C P[S]] \cap S \\
        \label{eq:efficient_points_definition}
        =&\, \set{\bar{x} \in S}{\nexists\, x \in S \colon P x \lneq_C P \bar{x}}
        \!\comma
    \end{align}
    where $P^{-1}[A] = \inlineset{x \in \R^n}{\exists\, z \in A : z = P x}$ denotes the preimage of~$A \subseteq \R^q$ under~$P$.
    The transposed rows of~$P$ are referred to as the \emph{objectives} of the program.
    A VLP is called \emph{multi-objective linear program} (MOLP) if one minimizes w.r.t.\ the natural order.
    For an MOLP, efficient points can be considered as the best compromises over the objectives.
    Note that maximization problems can be simply obtained by either considering $-P$ as the objective matrix or $-C$ as the ordering cone.

    In this article, we use the solution concept for VLPs from \citet{Loehne_Weissing_2017,Loehne_2011}, defined as follows.
    A pair $(\bar{S}, \hat{S})$ of nonempty sets is called an \emph{infimizer} of~\Pref\ if $\bar{S}$ consists of feasible points, $\hat{S}$ consists of directions of reccession of $S$ (also called feasible directions), and it generates the upper image in the sense that
    \begin{equation}
        \label{eq:definition_infimizer}
        \conv\!\rdb{P\!\sqb{\bar{S}}} + \cone\!\rdb{P\!\sqb{\hat{S}}} + C \supseteq P[S] + C \period
    \end{equation}
    Here, \enquote{$\conv$} stands for the convex hull, and \enquote{$\cone$} denotes the conic hull (also called the nonnegative hull).
    If, additionally,
    \begin{equation}
        \label{eq:definition_solution}
        P\!\sqb{\bar{S}} \subseteq \Min_C P[S] \qq{and} P\!\sqb{\hat{S}} \subseteq \Min_C P\!\sqb{\recc S} \!\comma
    \end{equation}
    where \enquote{$\recc$} denotes the recession cone~\citep[cf.][]{Rockafellar_1970}, then $(\bar{S}, \hat{S})$ is referred to as a \emph{solution} of~\Pref.

    As can be seen in \cref{fig:minimal_points_3}, it is possible that even nonempty compact sets do not have minimal points.
    This may occur if and only if the ordering cone is not pointed, i.e., if the ordering relation is not antisymmetric.
    To fix this, one can introduce the so-called generalized minimality.
    Note that there is no unique name for this concept.
    For instance, \citet{Jahn_2011} directly introduces minimality in the generalized form.
    We decided to distinguish both concepts, as \citet{Goepfert_Nehse_1990} do.

    \begin{figure}[hbt]
        \centering
        \begin{tikzpicture}[scale = 1.2]
    \node[below left] at (0, 0) {0};
    \draw[->] (-0.4, 0) -- (3.5, 0) node[below] {$z_1$};
    \draw[->] (0, -0.4) -- (0, 3) node[left] {$z_2$};

    \draw[set, shift = {(1, 0.5)}] (1, 0) -- (2, 0) -- (2, 2) -- (0, 2) -- (0, 1.5) -- (0.25, 0.5) -- cycle;
    \node[shift = {(1, 0.5)}] at (1.2, 1.2) {$A$};

    \draw[minimal points, dash pattern = on 1.91mm off 0.89mm, shift = {(1, 0.5)}] (0, 1.5) -- (0.25, 0.5);

    \draw[ordering cone line, shift = {(1, 0.5)}] (-0.15, 2.1) -- (0.4, -0.1);
    \fill[ordering cone fill, shift = {(1, 0.5)}] (-0.15, 2.1)+(-0.6,-0.15) -- (-0.15, 2.1) -- (0.4, -0.1) -- +(-0.6,-0.15);

\end{tikzpicture}
        \caption{
            An example for an ordering cone (shaded) that generates a preorder but not a partial order.
            Here, $A \subseteq \R^2$ has no minimal points.
            The set of generalized minimal points is the dashed edge of~$A$.
        }
        \label{fig:minimal_points_3}
    \end{figure}

    For a nonempty subset $A \subseteq \R^q$, a point~$\bar{z} \in A$ is called \emph{generalized minimal} w.r.t.\ $\leq_C$ in~$A$ if
    \begin{equation}
        \label{eq:generalized_minimal_points}
        \forall\, z \in A \colon
        z \leq_C \bar{z} \Longrightarrow \bar{z} \leq_C z
        \period
    \end{equation}
    The set of all generalized minimal points of~$A$ is denoted by $\gMin_C(A)$.
    Analogous to \cref{eq:efficient_points_definition}, one can extend VLPs to determine generalized efficient points, i.e.,
    \begin{equation}
        \label{eq:generalized_efficient_points_definition}
        \gEffPref
        \coloneq P^{-1}\!\sqb{\gMin_C P[S]} \cap S
        \period
    \end{equation}

    It follows directly from the definition that minimal points are always generalized minimal points.
    Additionally, one can easily see that generalized minimality coincides with minimality if $C$ is pointed as $\leq_C$ then is antisymmetric.
    The set of generalized minimal points can alternatively be expressed as
    \begin{align}
        \notag
        \gMin_C(A)
        &= \set{\bar{z} \in A}{A \cap (\bar{z} - C \setminus (C \cap (-C))) = \emptyset} \\
        \label{eq:generalized_minimal_points_set}
        &= \set{\bar{z} \in A}{A \cap (\bar{z} - C \setminus (-C)) = \emptyset}
        \!\period
    \end{align}
    By this generalization, the elements of the dashed set in \cref{fig:minimal_points_3} can be considered minimal.


    \section{Reduction of the Number of Objectives}
    \label{sec:reduction_of_objectives}

    When determining efficient solutions of an MOLP
    \hypertarget{molp}{
        \begin{flalign}
            \label{eq:molp}
            &\text{(MOLP)}
            &\min \quad P x \qq{s.t.} x \in S
            \comma
            &&
        \end{flalign}
    }%
    \newcommand*{\molpref}{(\hyperlink{molp}{MOLP})}%
    \newcommand*{\Effmolpref}{\Eff(\hyperlink{molp}{\text{MOLP}})}%
    \newcommand*{\gEffmolpref}{\gEff(\hyperlink{molp}{\text{MOLP}})}%
    many algorithms primarily operate in the image space of the problem (see \cref{sec:introduction} for examples).
    In practice, one can observe that the computation time grows rapidly when increasing the number of objectives.
    Consequently, it is of huge interest to reduce this number.

    If the objective matrix does not have a full row rank, i.e., $k \coloneq \rk(P) < q$, the objectives are linearly dependent.
    This means that the set of objectives contains less information than initially assumed.
    Hence, one might ask whether the number of objectives can be reduced to~$k$ without losing information, that is, without altering the solutions of the MOLP.
    One approach could be to remove $q - k$~objectives such that the remaining $k$~objectives are linearly independent.
    As discussed in \cref{sec:nonessential_objectives}, one can withdraw an objective generally if it is a conic (i.e., nonnegative) combination of the others.
    \Cref{ex:objective_removal}, however, shows that this is not true for non-conic combinations as it significantly changes the solution set.

    \begin{example}
        \label{ex:objective_removal}
        Consider the MOLP
        \hypertarget{molpex}{
            \begin{flalign}
                \label{eq:objective_removal_MOLP}
                &\text{($\text{MOLP}_1$)}
                &\min \quad
                \begin{pmatrix}
                    -1 &  0 &  0 \\
                    0 &  1 &  0 \\
                    1 & -1 &  0
                \end{pmatrix}
                x
                \qq{s.t.} x \in [0, 1]^3
                \period
                &&
            \end{flalign}
        }%
        \newcommand*{\Effmolpexref}{\Eff(\hyperlink{molpex}{\text{$\text{MOLP}_1$}})}%
        \newcommand*{\gEffmolpexref}{\gEff(\hyperlink{molpex}{\text{$\text{MOLP}_1$}})}%
        It holds that $\Effmolpexref = [0, 1]^3$.
        Removing the last objective yields a full-rank objective matrix:
        \hypertarget{molpexrem}{
            \begin{flalign}
                \label{eq:objective_removal_MOLP_removed}
                &\text{($\text{MOLP}_2$)}
                &\min \quad
                \begin{pmatrix}
                    -1 &  0 &  0 \\
                    0 &  1 &  0
                \end{pmatrix}
                x
                \qq{s.t.} x \in [0, 1]^3
                \period
                &&
            \end{flalign}
        }%
        \newcommand*{\molpexremref}{(\hyperlink{molpexrem}{$\text{MOLP}_2$})}%
        \newcommand*{\Effmolpexremref}{\Eff(\hyperlink{molpexrem}{\text{$\text{MOLP}_2$}})}%
        \newcommand*{\gEffmolpexremref}{\gEff(\hyperlink{molpexrem}{\text{$\text{MOLP}_2$}})}%
        However, $\Effmolpexremref = \{1\} \times \{0\} \times [0, 1] \neq \Effmolpexref$.
    \end{example}
    \newcommand*{\molpexref}{(\hyperlink{molpex}{$\text{MOLP}_1$})}%

    For this reason, another approach to reduce the number of objectives is necessary in order to completely remove linear dependence.
    As a first step, we decompose the objective matrix~$P$ into a product of two matrices~$L \in \R^{q \times k}$ and~$R \in \R^{k \times n}$ where $\rk(L) = \rk(R) = k$, i.e., $P = LR$.
    This is always possible, e.g., by using Gauß’ algorithm or a singular value decomposition.
    Note that the decomposition is not unique as it depends on the method employed to obtain it.
    Moreover, $L$ is always injective since it has full column rank, and $R$ is always surjective due to its full row rank.
    Now, we define the set
    \begin{equation}
        \label{eq:matrix_cone}
        C \coloneq \set{z \in \R^k}{Lz \geq 0}
    \end{equation}
    which is a nonempty convex cone as $0 \in C$, and $z \in C$ and $\lambda \geq 0$ imply $L (\lambda z) = \lambda (Lz) \geq 0$, that is, $\lambda z \in C$.
    Since $C \cap (-C) = \inlineset{z \in \R^k}{Lz = 0} = \clb{0}$ holds by the linear independence of the columns of~$L$, $C$ is pointed.
    Therefore, $C$ defines a partial order~$\leq_C$.

    Using the decomposition of~$P$, the linearity of~$L$, and \cref{eq:matrix_cone}, we obtain the following equivalence for any $x^1, x^2 \in \R^n$:
    \begin{align}
        \notag
        P x^1 \leq P x^2
        &\;\Longleftrightarrow\; L R x^1 \leq L R x^2 \\
        \notag
        &\;\Longleftrightarrow\; L (R x^2 - R x^1) \geq 0 \\
        \label{eq:matrix_partial_order_P}
        &\;\Longleftrightarrow\; R x^1 \leq_C R x^2
        \period
    \end{align}
    Thus, we can compare two feasible points under $R$ in $\R^k$ instead of $P$ in $\R^q$.
    Using $R$ as the objective matrix and $C$ as the ordering cone of a new VLP
    \hypertarget{vlp}{
        \begin{flalign}
            \label{eq:vlp_from_molp}
            &\text{(VLP)}
            &\conemin{C} \quad R x \qq{s.t.} x \in S
            \comma
            &&
        \end{flalign}
    }%
    \newcommand*{\vlpref}{(\hyperlink{vlp}{VLP})}%
    \newcommand*{\Effvlpref}{\Eff(\hyperlink{vlp}{\text{VLP}})}%
    \newcommand*{\gEffvlpref}{\gEff(\hyperlink{vlp}{\text{VLP}})}%
    the number of objectives decreases from $q$ to $k$.

    The following theorem states that \molpref\ and \vlpref\ are equivalent in the sense that their sets of efficient points coincide.
    Furthermore, the minimal points are related just by applying~$L$.
    As discussed in \cref{sec:introduction}, this theorem has already been stated and proven in the literature.

    \begin{theorem}
        \label{thm:equivalence_molp_vlp}
        Consider the problem~\molpref\ with a decomposition of the objective matrix~$P = LR$ into the matrices~$L \in \R^{q \times k}$ and~$R \in \R^{k \times n}$ where $k = \rk(P) = \rk(L) = \rk(R)$.
        Then, for the problem~\vlpref\ the following statements hold:
        \begin{equation}
            \label{eq:equivalence_molp_vlp_eff}
            \Effmolpref = \Effvlpref
        \end{equation}
        and
        \begin{equation}
            \label{eq:equivalence_molp_vlp_min}
            \Min P[S] = L[\Min_C R[S]]
            \period
        \end{equation}
    \end{theorem}
    \begin{proof}
        To prove \cref{eq:equivalence_molp_vlp_eff}, take some $\bar{x} \in \Effvlpref$.
        According to \cref{eq:efficient_points_definition}, there is no point~$x \in S$ with both $R x \neq R \bar{x}$ and $R x \leq_C R \bar{x}$.
        Due to \cref{eq:matrix_partial_order_P}, one can rewrite the inequality as $P x \leq P \bar{x}$.
        As, by assumption, $L$ is injective, $R x \neq R \bar{x}$ is equivalent to $L R x \neq L R \bar{x}$, that is, $P x \neq P \bar{x}$.
        The equivalence of $\bar{x} \in \Effvlpref$ and $\bar{x} \in \Effmolpref$ then follows from \cref{eq:efficient_points_definition}, again.

        \Cref{eq:equivalence_molp_vlp_min} follows directly from \cref{eq:equivalence_molp_vlp_eff} as follows.
        With \cref{eq:efficient_points_definition} we get
        \begin{equation}
            \label{eq:equivalence_molp_vlp_minimal_points_1}
            \Min P[S] = P[\Effmolpref] = P[\Effvlpref]
            \period
        \end{equation}
        Applying $P = LR$ and, again, \cref{eq:efficient_points_definition} yields
        \begin{equation}
            \label{eq:equivalence_molp_vlp_minimal_points}
            P[\Effvlpref] = LR\!\sqb{R^{-1}[\Min_C R[S]] \cap S} = L[\Min_C R[S]]
            \period
        \end{equation}
        In order to show that $R R^{-1}$ can be omitted in the last equation, we employed $\Min_C R[S] \subseteq R[S]$.
        \qedhere
    \end{proof}

    Next to the aforementioned speedup in computing time, which we will discuss in more detail in \cref{sec:examples}, another crucial practical benefit arises from \cref{thm:equivalence_molp_vlp}.
    It becomes much easier for decision makers to get a feeling for the actual problem since the number of objectives decreases, and plotting the problem (e.g., via net diagrams) becomes simpler.
    In particular, the objectives’ linear dependence means that the objectives contain less information than one primarily has supposed.
    Due to linear dependencies, the image of the feasible set just lies in a $k$-dimensional subspace of $\R^q$, and the transformation into a VLP is nothing else than a projection onto $\R^k$.
    Nevertheless, particular linear combinations of objectives might be of special practical interest---although they actually are redundant for the solution, one can output them additionally.

    Decision makers become able to interpret the problem \vlpref\ by considering the problem
    \hypertarget{molpinterpretation}{
        \begin{flalign}
            \label{eq:molp_interpretation}
            &\text{($\text{MOLP}^\circ$)}
            &\min \quad L z \qq{s.t.} z \in R[S]
            \comma
            &&
        \end{flalign}
    }%
    \newcommand*{\molpinterpretationref}{(\hyperlink{molpinterpretation}{$\text{MOLP}^\circ$})}%
    \newcommand*{\Effmolpinterpretationref}{\Eff(\hyperlink{molpinterpretation}{\text{MOLP}^\circ})}%
    \newcommand*{\gEffmolpinterpretationref}{\gEff(\hyperlink{molpinterpretation}{\text{MOLP}^\circ})}%
    that is equivalent to \molpref\ and, due to \cref{thm:equivalence_molp_vlp}, also to \vlpref.
    A feasible point~$\bar{x}$ is efficient for \vlpref, i.e., $\bar{z} = R \bar{x}$ is $C$-minimal, if and only if $\bar{z} = R \bar{x}$ is efficient for \molpinterpretationref.
    Hence, at an efficient point, it is not possible to decrease any objective~$\ell^i$ over~$R[S]$ without increasing another objective~$\ell^j$, where $\ell^1, \ldots, \ell^q$ denote the transposed rows of $L$.
    By construction of~$C$, the objectives~$\ell^1, \ldots, \ell^q$ are normals to the supporting hyperplanes of~$C$.
    As we discuss in \cref{sec:nonessential_objectives}, only the objectives orthogonal to some extremal direction of~$C$ are needed to be considered here.
    Since the rows of $L$ and $P$ correspond one-to-one due to the injectivity of $R^\transposed$, the behavior of the objectives of \molpinterpretationref\ translates directly to that of \molpref.
    That is, minimizing an objective~$\ell^i$ over~$R[S]$ minimizes the corresponding objective~$p^i$ over~$S$, and vice versa.
    Decision makers benefit from this as they can vary the objectives of \molpinterpretationref\ over a $k$-dimensional space instead of those of \molpref\ over an $n$-dimensional space.
    This relationship can also be used to visualize \vlpref; an example is illustrated in \cref{fig:example_interpretation}.

    \begin{figure}[hbt]
        \centering
        \begin{tikzpicture}[scale = 1.5]
    \node[below left] at (0, 0) {0};
    \draw[->] (-0.4/1.5, 0) -- (4.4, 0) node[below] {$z_1 = (r^1)^\transposed x$};
    \draw[->] (0, -0.4/1.5) -- (0, 3.15) node[left] {$z_2 = (r^2)^\transposed x$};

    \coordinate (A) at (1.25, 0.75);
    \coordinate (B) at (2.25, 0.5);
    \coordinate (C) at (2.75, 1.5);
    \coordinate (D) at (2.25, 2.25);
    \coordinate (E) at (1.5, 2.25);
    \coordinate (F) at (0.75, 1.5);
    \draw[set] (A) -- (B) -- (C) -- (D) -- (E) -- (F) -- cycle;
    \node at (1.9, 1.6) {$R[S]$};

    \draw[upper image] (B) -- +(2, 1);
    \draw[upper image] (F) -- +(0.75, 1.5);
    \node at (3.5, 1.75) {$R[S] + C$};

    \coordinate (P) at ($0.6*(A) + 0.4*(F)$);
    \draw[ordering cone] (P)+($sqrt(5)/5*(-2, -1)$) -- (P) -- +($sqrt(5)/5*(-1, -2)$) node[right]{$\bar{z} - C$};

    \draw[minimal points] (F) -- (A) -- (B);
    \draw[point scaled] (P) node[right]{$\bar{z}$} circle;

    \draw[objective] (F) +($sqrt(5)/20*(-1, -2)$) -- +($sqrt(5)/20*(1, 2)$);
    \draw[objective, ->] (F) -- node[above right = -2pt]{$\ell^1$} +($sqrt(5)/10*(2, -1)$);
    \draw[objective] (B) +($sqrt(5)/20*(-2, -1)$) -- +($sqrt(5)/20*(2, 1)$);
    \draw[objective, ->] (B) -- node[above right = -2pt]{$\ell^2$} +($sqrt(5)/10*(-1, 2)$);

\end{tikzpicture}
        \caption{%
            Visualization of a \protect\vlpref\ in the image space and the corresponding problem \protect\molpinterpretationref\ in the space of variables.
            The arrows show two objectives $\ell^1$ and $\ell^2$ at their minimum over $R[S]$.
            The dashed lines indicate the upper image $R[S] + C$ of \protect\vlpref, the thick ones show the minimal points of \protect\vlpref, which coincide with the efficient points of \protect\molpinterpretationref.
            By $r^1$ and $r^2$, the transposed rows of $R$ are denoted.
        }
        \label{fig:example_interpretation}
    \end{figure}

    One can implement the method into a solution algorithm for MOLP using the following distinction of cases, provided the rank of $P$ has already been calculated.
    If $k = n < q$, a trivial decomposition (i.e., $L = P$ and $R$ being the identity matrix in $\R^{n \times n}$) leads to a reduction in the number of objectives.
    If $P$ has full rank and $k = q \leq n$, the procedure has no advantage and can be omitted.
    In any other case in which $P$ does not have a full rank, one should always compute a decomposition of $P$ and solve the corresponding VLP.

    \begin{remark}
        \label{rem:cone_dimension}
        The dimension of the ordering cone is not uniquely determined by the rank and the size of the objective matrix.
        For instance, consider the matrices
        \begin{equation}
            \label{eq:cone_dimension_example}
            L_1 =
            \begin{pmatrix}
                1 & 0 & 1 \\
                0 & 1 & 1
            \end{pmatrix}^{\!\transposed}
            \qq{and}
            L_2 =
            \begin{pmatrix}
                1 & 0 & -1 \\
                0 & 1 & -1
            \end{pmatrix}^{\!\transposed}
            \!\period
        \end{equation}
        Note that these matrices are obtained by decomposing some rank-\num{2} objective matrices, which can be constructed by multiplying $L_1$ and $L_2$ by some $\num{2} \times n$ matrix from the right.
        The resulting ordering cones are $C_1 = \R^2_+$ and $C_2 = \{0\}$, with dimensions~\num{2} and~\num{0}, respectively.

        However, the objective matrix does determine the dimension of the ordering cone uniquely.
        Consider the cone~$C_P = \inlineset{x \in \R^n}{P x \geq 0}$ induced by the objective matrix.
        We can split the representation of~$C_P$ into implicit equations (that is, inequalities that are fulfilled with equality for all $x \in C_P$) and proper inequalities (i.e., inequalities~$p^\transposed x \geq 0$ for which there exists $\bar{x} \in C_P$ such that $p^\transposed \bar{x} > 0$).
        Denoting the rows of~$P$ corresponding to the implicit equations by~$P_\text{eq} \in \R^{s \times n}$ for some $s \in \N_0$ and those corresponding to proper inequalities by~$P_\text{\,ineq} \in \R^{(q - s) \times n}$, we obtain
        \begin{equation}
            \label{eq:cone_dimension_C_P}
            C_P = \set{x \in \R^n}{P_\text{eq} x = 0} \cap \set{x \in \R^n}{P_\text{\,ineq} x \geq 0} \!\period
        \end{equation}
        Due to the surjectivity of the matrix~$R$ in the decomposition~$P = LR$, the representation of the ordering cone from \cref{eq:matrix_cone} can be split similarly where $P_\text{eq} = L_\text{eq} R$ and $P_\text{\,ineq} = L_\text{ineq} R$:
        \begin{equation}
            \label{eq:cone_dimension_C_L}
            C = \set{x \in \R^n}{L_\text{eq} x = 0} \cap \set{x \in \R^n}{L_\text{ineq} x \geq 0} \!\period
        \end{equation}
        Since the linear hull of~$C$ is exactly the first part of this intersection, the dimension of~$C$ is
        \begin{equation}
            \label{eq:cone_dimension}
            \dim(C) = \dim(\ker(L_\text{eq})) = k - \rk(L_\text{eq}) = k - \rk(P_\text{eq}) \comma
        \end{equation}
        where we set the rank of a zero-row matrix to zero and, again, used the surjectivity of~$R$ for the last equation.
        Moreover, as the linear hull of a polyhedral cone is the orthogonal complement of the lineality space of its polar cone~\citep[Theorem~14.6]{Rockafellar_1970}, $\rk(L_\text{eq})$ and $\rk(P_\text{eq})$ coincide with the dimension of the lineality spaces of~$\cone(L^\transposed)$ and~$\cone(P^\transposed)$, respectively.
        This means that the presence of objectives whose conic hull forms a nontrivial subspace of~$\R^n$ induces a dimension reduction of the ordering cone of~\vlpref.

        In conclusion, the dimension of the ordering cone can be calculated by determining the rank of~$P_\text{eq}$ or~$L_\text{eq}$, or by determining the dimension of the lineality spaces of~$\cone(P^\transposed)$ or~$\cone(L^\transposed)$.
        Additionally, VLP solution software must be able to handle ordering cones that are not solid, that is, cones that do not have a full dimension or, equivalently, have an empty interior.
    \end{remark}

    \begin{remark}
        \label{rem:extension}
        The method is not limited to MOLP but can be applied to all multi-objective programs with linear objective functions, independent of the structure of the feasible set and the ordering cone.
        In particular, consider the following VLP
        \hypertarget{vlpmethodstart}{
            \begin{flalign}
                \label{eq:vlpmethodstart}
                &\text{($\text{VLP}'$)}
                &\conemin{C'} \quad P x \qq{s.t.} x \in S
                &&
            \end{flalign}
        }%
        \newcommand*{\vlpmethodstartref}{(\hyperlink{vlpmethodstart}{$\text{VLP}'$})}%
        \newcommand*{\Effvlpmethodstartref}{\Eff(\hyperlink{vlpmethodstart}{\text{VLP}'})}%
        \newcommand*{\gEffvlpmethodstartref}{\gEff(\hyperlink{vlpmethodstart}{\text{VLP}'})}%
        where $C' = \inlineset{z' \in \R^q}{B z' \geq 0}$ for some $B \in \R^{r \times q}$, and $P \in \R^{q \times n}$ with $\rk(P) = k$.
        Take a decomposition of the objective matrix into the matrices~$L \in \R^{q \times k}$ and~$R \in \R^{k \times n}$ where $k = \rk(P) = \rk(L) = \rk(R)$.
        Then, for the problem
        \hypertarget{vlpmethodend}{
            \begin{flalign}
                \label{eq:vlpmethodend}
                &\text{($\text{VLP}''$)}
                &\conemin{C''} \quad R x \qq{s.t.} x \in S
                &&
            \end{flalign}
        }%
        \newcommand*{\vlpmethodendref}{(\hyperlink{vlpmethodend}{$\text{VLP}''$})}%
        \newcommand*{\Effvlpmethodendref}{\Eff(\hyperlink{vlpmethodend}{\text{VLP}''})}%
        \newcommand*{\gEffvlpmethodendref}{\gEff(\hyperlink{vlpmethodend}{\text{VLP}''})}%
        with ordering cone~$C'' = \inlineset{z'' \in \R^k}{B L z'' \geq 0}$, the following statements hold:
        \begin{equation}
            \label{eq:equivalence_vlp_vlp_eff}
            \Effvlpmethodstartref = \Effvlpmethodendref
        \end{equation}
        and
        \begin{equation}
            \label{eq:equivalence_vlp_vlp_min}
            \Min_{C'} P[S] = L[\Min_{C''} R[S]]
            \period
        \end{equation}
        This follows similarly to the proof of \cref{thm:equivalence_molp_vlp}.
    \end{remark}

    In the remaining part of this section, we want to discuss the requirement of injectivity of~$L$ and provide a generalization of \cref{thm:equivalence_molp_vlp}.
    Decomposing~$P$ into $L \in \R^{q \times k}$ and $R \in \R^{k \times n}$ such that $\rk(P) \leq \rk(L) < k$ yields a cone~$C$ that is not pointed since $C \cap (-C) = \ker(L)$ is a proper superset of~$\{0\}$.
    The proof of \cref{thm:equivalence_molp_vlp} shows that $\Effvlpref \subseteq \Effmolpref$ always holds as the injectivity of $L$ is not necessary for the proof of this inclusion.
    One can easily construct examples where the equality is not true.
    For instance, we can take a bounded feasible set~$S$ and a matrix~$L$ such that $C \cap (-C)$ is parallel to a non-vertex face of $R[S]$.
    Then, if $C$ has the right orientation, no point of this face is $C$-minimal due to the missing antisymmetry of the ordering relation, cf.\ \cref{fig:minimal_points_3}.
    Moreover, there are no other minimal points, so $\Effvlpref = \emptyset$.
    As $S$ is compact, $P[S] = L[R[S]]$ has minimal points w.r.t.\ the natural order.
    Thus, $\Effmolpref \neq \emptyset$.
    In conclusion, the injectivity of $L$ is necessary for \cref{thm:equivalence_molp_vlp}.

    In this sketch of an example, considering generalized minimal points will fix the problem.
    Indeed, this approach resolves it in general.
    We will prove this in \cref{thm:equivalence_molp_vlp_generalized} below.

    \begin{theorem}
        \label{thm:equivalence_molp_vlp_generalized}
        Consider the problem~\molpref\ with a decomposition of the objective matrix~$P = LR$ into the matrices~$L \in \R^{q \times k}$ and~$R \in \R^{k \times n}$ where $k \geq \rk(P)$.
        Then, for the problem~\vlpref\ the following statements hold:
        \begin{equation}
            \label{eq:equivalence_molp_vlp_generalized_eff}
            \Effmolpref = \gEffvlpref
        \end{equation}
        and
        \begin{equation}
            \label{eq:equivalence_molp_vlp_generalized_min}
            \Min P[S] = L[\gMin_C R[S]]
            \period
        \end{equation}
    \end{theorem}
    \begin{proof}
        To show the first equation, take some $\bar{x} \in \Effmolpref$.
        Because the ordering cone $\R_+^q$ of \molpref\ is pointed, the set of efficient points coincides with the set of generalized efficient points, as discussed in \cref{sec:problem_setting_preliminaries}.
        Thus, $\bar{x} \in \Effmolpref$ is equivalent to $\bar{x} \in \gEffmolpref$.
        According to \cref{eq:generalized_minimal_points}, for any $x \in S$ satisfying $P x \leq P \bar{x}$ follows $P \bar{x} \leq P x$.
        Note that \cref{eq:matrix_partial_order_P} also holds if $L$ is not injective since only the definitions of $C$ and $\leq_C$ were used there.
        Applying \cref{eq:matrix_partial_order_P}, we obtain that $R x \leq_C R \bar{x}$ implies $R \bar{x} \leq_C R x$ for all $x \in S$.
        This is equivalent to $\bar{x} \in \gEffvlpref$.
        The second equation follows similarly to \cref{thm:equivalence_molp_vlp}.
        \qedhere
    \end{proof}

    \cref{thm:equivalence_molp_vlp} can be viewed as a corollary of \cref{thm:equivalence_molp_vlp_generalized} since efficient and generalized efficient points coincide for pointed ordering cones.
    \cref{rem:extension} also applies to a statement similar to \cref{thm:equivalence_molp_vlp_generalized} regarding the equivalence of suitable VLP.

    \begin{remark}
        \label{rem:equivalence_molp_vlp_generalized_comparison_to_Thoai_2012}
        In \citet[Proposition~1]{Thoai_2012}, the special case of $L = P$ and $R$ being the identity matrix was proven through a slight modification of the ordering cone:
        \begin{equation}
            \label{eq:matrix_cone_generalized}
            \tilde{C}
            \coloneq \set{z \in \R^k}{P z \gneq 0} \cup \clb{0}
            = (C \setminus (-C)) \cup \clb{0}
            \period
        \end{equation}
        That is, the lineality directions of $C$ are excluded.
        The notion of minimality induced by~$\tilde{C}$ coincides with the generalized minimality w.r.t.\ $\leq_C$, introduced in \cref{eq:generalized_minimal_points}.
        From a theoretical point of view, it is possible to treat the objective matrix~$P$ in two ways: by using $\tilde{C}$ as the ordering cone, as~\citet{Thoai_2012} does, or the concept of generalized minimality, as we do in the proof of the extended \cref{thm:equivalence_molp_vlp}.
    \end{remark}

    From a practical point of view, \cref{thm:equivalence_molp_vlp_generalized} is not relevant since the number of objectives of~\vlpref\ is not minimal in this case.
    However, the generalization is a valuable proving technique demonstrated in the next section.


    \section{Nonessential Objectives}
    \label{sec:nonessential_objectives}

    It is possible that some objectives of \molpref\ are not relevant for determining efficient points---they are obsolete.
    \Cref{subfig:example_nonessential_objective_1} gives an example where one can remove, for instance, the objective $p^2$ without affecting the set of efficient points.
    This concept of so-called nonessential (or redundant) objectives has been developed by Thomas Gal and Heiner Leberling~\citep{Gal_Leberling_1977,Gal_1980}.
    In this section we aim to present this concept based on \citet{Gal_Leberling_1977,Gal_1980,Thoai_2012} and relate it to the previous section.
    First, the following introduces the terms of nonessential objective as well as nonessential system of objectives.

    \begin{figure}[hbt]
        \centering
        \begin{subfigure}{0.49\linewidth}
            \centering
            \begin{tikzpicture}[scale = 1.5]
    \node[below left] at (0, 0) {0};
    \draw[->] (-0.4/1.5, 0) -- (3.5/1.5, 0) node[below] {$x_1$};
    \draw[->] (0, -0.4/1.5) -- (0, 3.5/1.5) node[left] {$x_2$};

    \coordinate (A) at (0, 0);
    \coordinate (B) at (2, 0);
    \coordinate (C) at (2, 1);
    \coordinate (D) at (1, 2);
    \coordinate (E) at (0, 2);
    \draw[set] (A) -- (B) -- (C) -- (D) -- (E) -- cycle;
    \node at (0.7, 0.7) {$S$};

    \draw[minimal points] (C) -- (D) -- (E);

    \draw[objective] (E) +($sqrt(5)/20*(-2, -1)$) -- +($sqrt(5)/20*(2, 1)$);
    \draw[objective, ->] (E) -- node[right = 2pt]{$p^1$} +($sqrt(5)/10*(1, -2)$);
    \draw[objective] (D) +($sqrt(5)/20*(2, -1)$) -- +($sqrt(5)/20*(-2, 1)$);
    \draw[objective, ->] (D) -- node[left = -2pt]{$p^2$} +($sqrt(5)/10*(-1, -2)$);
    \draw[objective] (C) +($sqrt(13)/52*(2, -3)$) -- +($sqrt(13)/52*(-2, 3)$);
    \draw[objective, ->] (C) -- node[below]{$p^3$} +($sqrt(13)/26*(-3, -2)$);
    \draw[objective] (C) +($sqrt(17)/68*(1, -4)$) -- +($sqrt(17)/68*(-1, 4)$);
    \draw[objective, ->] (C) -- node[above left = -2pt]{$p^4$} +($sqrt(17)/34*(-4, -1)$);

\end{tikzpicture}
            \caption{}
            \label{subfig:example_nonessential_objective_1}
        \end{subfigure}
        \begin{subfigure}{0.49\linewidth}
            \centering
            \begin{tikzpicture}[scale = 1.5]
    \node[below left] at (0, 0) {0};
    \draw[->] (-0.4/1.5, 0) -- (3.5/1.5, 0) node[below] {$x_1$};
    \draw[->] (0, -0.4/1.5) -- (0, 3.5/1.5) node[left] {$x_2$};

    \coordinate (A) at (0, 0);
    \coordinate (B) at (2, 0);
    \coordinate (C2) at (1.5, 1.5);
    \coordinate (C1) at (2, 0.5);
    \coordinate (D) at (1, 2);
    \coordinate (E) at (0, 2);
    \draw[set] (A) -- (B) -- (C1) -- (C2) -- (D) -- (E) -- cycle;
    \node at (0.7, 0.7) {$S$};

    \draw[minimal points] (C1) -- (C2) -- (D) -- (E);

    \draw[objective] (E) +($sqrt(5)/20*(-2, -1)$) -- +($sqrt(5)/20*(2, 1)$);
    \draw[objective, ->] (E) -- node[right = 2pt]{$p^1$} +($sqrt(5)/10*(1, -2)$);
    \draw[objective] (D) +($sqrt(5)/20*(2, -1)$) -- +($sqrt(5)/20*(-2, 1)$);
    \draw[objective, ->] (D) -- node[left = -2pt]{$p^2$} +($sqrt(5)/10*(-1, -2)$);
    \draw[objective] (C2) +($sqrt(13)/52*(2, -3)$) -- +($sqrt(13)/52*(-2, 3)$);
    \draw[objective, ->] (C2) -- node[below]{$p^3$} +($sqrt(13)/26*(-3, -2)$);
    \draw[objective] (C1) +($sqrt(17)/68*(1, -4)$) -- +($sqrt(17)/68*(-1, 4)$);
    \draw[objective, ->] (C1) -- node[above left = -2pt]{$p^4$} +($sqrt(17)/34*(-4, -1)$);

\end{tikzpicture}
            \caption{}
            \label{subfig:example_nonessential_objective_2}
        \end{subfigure}
        \caption{
            Visualization in the space of variables of two MOLPs for which some objectives can be omitted without affecting the set of efficient points.
            In both cases, the objectives $p^1, \ldots, p^4$ (shown by the arrows at the position of minimum) are the same, where in (b) there is an additional constraint.
            The thick lines indicate the efficient points of the MOLPs.
        }
        \label{fig:example_nonessential_objective}
    \end{figure}

    Consider the problem \molpref\ and $\Pcal = \{p^1, \ldots, p^q\}$ as the set of its objectives.
    The objective $p^{\,j} \in \Pcal$ is called \emph{nonessential} if for the problem
    \hypertarget{molpnonessential}{
        \begin{flalign}
            \label{eq:molp_nonessential}
            &\text{($\text{MOLP}^*$)}
            &\min \quad
            \begin{pmatrix}
                p^1 & \cdots & p^{\,j - 1} & p^{\,j + 1} & \cdots & p^q
            \end{pmatrix}^{\!\transposed}
            x
            \qq{s.t.} x \in S
            \comma
            &&
        \end{flalign}
    }%
    \newcommand*{\molpnonessentialref}{(\hyperlink{molpnonessential}{$\text{MOLP}^*$})}%
    \newcommand*{\Effmolpnonessentialref}{\Eff(\hyperlink{molpnonessential}{\text{MOLP}^*})}%
    \newcommand*{\gEffmolpnonessentialref}{\gEff(\hyperlink{molpnonessential}{\text{MOLP}^*})}%
    $\Effmolpnonessentialref = \Effmolpref$ holds.
    Otherwise, $p^{\,j}$ is \emph{essential}.
    A subset $\Pcal^* \subseteq \Pcal$ is called a \emph{nonessential system} of objectives if all objectives in~$\Pcal^*$ can be removed from \molpref\ simultaneously in the above sense, i.e., without altering $\Effmolpref$.

    It is important to note that a system of nonessential objectives does not have to be a nonessential system of objectives.
    For the example in \cref{subfig:example_nonessential_objective_1}, $\{p^2, p^3, p^4\}$ and $\{p^3, p^4\}$ are not nonessential, even though the individual objectives are nonessential themselves.
    However, the systems $\{p^2, p^3\}$ and $\{p^2, p^4\}$ are nonessential.

    In general, it depends on the feasible set which objectives or systems of objectives, respectively, are nonessential.
    For instance, consider the problem in \cref{subfig:example_nonessential_objective_2} that has the same objectives but a slightly modified feasible set.
    Here, $p^4$ is essential.
    This dependency also holds for the concepts of absolutely and relatively redundant systems of objectives, cf.\ \citet{Gal_Leberling_1977}.
    It is difficult to obtain a geometrical impression of the feasible set, that is, to know its vertices and directions of recession.
    If we wanted to find nonessential objectives depending on the feasible set, we would need to know the form of the feasible set in advance.
    For this, we would perform a vertex enumeration in $\R^n$.
    As $n$ is usually (much) larger than $q$, this would significantly increase computing time---contradicting our intention of decreasing it.

    Therefore, a statement not depending on the feasible set is more valuable.
    The following theorem shows that a statement independent of the feasible set can be derived for objectives that can be represented by conic combinations of the others.
    This result has already been investigated in previously mentioned studies in~\citet{Gal_Leberling_1977,Thoai_2012}.
    Here, we present a proof that uses the reformulation of the MOLP into an equivalent VLP, employing the concept of generalized minimality from \cref{sec:reduction_of_objectives}.
    This proof is similar to the one in~\citet[Corollary~1]{Thoai_2012} that uses the cone from \cref{eq:matrix_cone_generalized}, see \cref{rem:nonessential_objective_conic_combination_comparison_to_Thoai_2012} for more details.
    In~\citet{Gal_Leberling_1977}, \cref{enm:ne_objective_conic_combination} of the theorem has been proven via the concept of scalarization, cf.\ Theorem~4.2 in~\citet{Loehne_2011}.
    \Cref{enm:ne_system_conic_combination} then follows from iterating \cref{enm:ne_objective_conic_combination} over the elements of $\Pcal^*$.
    Note that we do not use the concept of scalarization in our proof.

    \begin{theorem}
        \label{thm:nonessential_objective_conic_combination}
        Consider the problem \molpref, and let $\Pcal$ be the set of its objectives.
        \begin{enumerate}
            \item The objective $p \in \Pcal$ is nonessential if \label{enm:ne_objective_conic_combination}
            \begin{equation}
                \label{eq:ne_objective_conic_combination}
                p \in \cone(\Pcal \setminus \clb{p})
                \period
            \end{equation}
            \item The subset $\Pcal^* \subseteq \Pcal$ is a nonessential system of objectives if \label{enm:ne_system_conic_combination}
            \begin{equation}
                \label{eq:ne_system_conic_combination}
                \forall\, p \in \Pcal^* \colon
                p \in \cone(\Pcal \setminus \Pcal^*)
                \period
            \end{equation}
        \end{enumerate}
        \vspace{0pt}
    \end{theorem}

    \begin{proof}
        As the objective in a nonessential system $\Pcal^* = \{p\}$ with only one element is, by definition, always nonessential, \cref{enm:ne_objective_conic_combination} follows directly from \cref{enm:ne_system_conic_combination}.
        Therefore, it suffices to show \cref{enm:ne_system_conic_combination}.

        Take a trivial decomposition of the objective matrix $P$, i.e., $L = P$ and $R$ being the $n$-dimensional identity matrix.
        From $L$, one obtains the cone $C$ via \cref{eq:matrix_cone} and thus the associated problem \vlpref\ according to \cref{eq:vlp_from_molp}.
        \Cref{thm:equivalence_molp_vlp_generalized} then yields $\gEffvlpref = \Effmolpref$.

        Now remove all rows from $P$ which are contained in $\Pcal^*$ and obtain the submatrix $P' \in \R^{\cramped{q'} \times n}$.
        As we will prove below, $P'$ also generates $C$.
        Thus, the VLP will not be altered when replacing $P$ by $P'$.
        However, one can apply \cref{thm:equivalence_molp_vlp_generalized} again to the simplified VLP.
        It remains an MOLP
        \hypertarget{molpreduced}{
            \begin{flalign}
                \label{eq:molp_reduced}
                &\text{($\text{MOLP}'$)}
                &\min \quad
                P' x \qq{s.t.} x \in S
                &&
            \end{flalign}
        }%
        \newcommand*{\molpreducedref}{(\hyperlink{molpreduced}{$\text{MOLP}'$})}%
        \newcommand*{\Effmolpreducedref}{\Eff(\hyperlink{molpreduced}{\text{MOLP}'})}%
        \newcommand*{\gEffmolpreducedref}{\gEff(\hyperlink{molpreduced}{\text{MOLP}'})}%
        with just $q - \#\Pcal^*$ of the $q$ original objectives, namely the elements from $\Pcal \setminus \Pcal^*$.
        Due to $\Effmolpreducedref = \gEffvlpref = \Effmolpref$, $\Pcal^*$ forms a nonessential system of objectives.

        It remains to show that the cone $C' = \{z \in \R^n \,\vert\, P' z \geq 0\}$ equals~$C$.
        Note that $P'$ only contains rows from $P$.
        Therefore, $C \subseteq C'$.
        To prove the converse inclusion, take some $z \in C'$.
        For $p \in \Pcal \setminus \Pcal^*$ (i.e., the rows of $P'$) one directly obtains $p^\transposed z \geq 0$.
        According to the assumption $\Pcal^* \subseteq \cone(\Pcal \setminus \Pcal^*)$, for $p \in \Pcal^*$ (i.e., all remaining rows of~$P$) there exists $\lambda \in \R_+^{\cramped{q'}}$ such that $p = (P')^\transposed \lambda$ holds.
        Thus, it follows $p^\transposed z = (\lambda^\transposed P') z = \lambda^\transposed (P' z) \geq 0$.
        Summarized, this yields $P z \geq 0$, and therefore $z \in C$.
        \qedhere
    \end{proof}

    \begin{remark}
        \label{rem:nonessential_objective_conic_combination_alternative_proof}
        Alternatively, one can prove the statement $C = C'$ via the polar cones $C^* = \cone(\Pcal)$ and $(C')^* = \cone(\Pcal \setminus \Pcal^*)$, respectively.
        Note that the bi-polar cone of a nonempty closed convex cone is the cone itself~\citep{Rockafellar_1970}.
        Therefore, it suffices to show $C^* = (C')^*$.
        $\Pcal \setminus \Pcal^* \subseteq \Pcal$ directly yields $C^* \supseteq (C')^*$.
        Conversely, the assumption $\Pcal^* \subseteq \cone(\Pcal \setminus \Pcal^*)$ together with the trivial inclusion $\Pcal \setminus \Pcal^* \subseteq \cone(\Pcal \setminus \Pcal^*)$ implies $\Pcal \subseteq \cone(\Pcal \setminus \Pcal^*)$, and thus $C^* \subseteq (C')^*$.
    \end{remark}

    \begin{remark}
        \label{rem:nonessential_objective_conic_combination_comparison_to_Thoai_2012}
        \citet[Definition~1]{Thoai_2012} introduces the notion of a \enquote{family of representative criteria} of \molpref.
        This is defined as \cref{eq:ne_system_conic_combination}.
        We can therefore reformulate \cref{enm:ne_system_conic_combination} of \cref{thm:nonessential_objective_conic_combination} as follows: \enquote{The subset $\Pcal^* \subseteq \Pcal$ is a nonessential system of objectives if it is a family of representative criteria.}

        Moreover, the statement $C = C'$ corresponds to Proposition~2 of~\citet{Thoai_2012}.
        While the inclusion $C \supseteq C'$ follows similarly, proving the converse inclusion $C \subseteq C'$ is easier.
        This is because the cones considered in~\citet{Thoai_2012} exclude the lineality directions of $C$ and $C'$, respectively, and therefore, an additional condition must be proven.
    \end{remark}

    According to \cref{thm:nonessential_objective_conic_combination}, for the examples in \cref{fig:example_nonessential_objective} the system $\{p^2, p^3\}$ is non-essential independent of the feasible set since $p^2$ and $p^3$ can be expressed as conic combinations of the other objectives $p^1$ and $p^4$.
    Note that \cref{thm:nonessential_objective_conic_combination} only gives a sufficient but not necessary condition~\citep{Gal_Leberling_1977}; in the example in \cref{subfig:example_nonessential_objective_1}, the objective $p^4$ is non-essential even though it is not a conic combination of the others.
    For a both necessary and sufficient condition, we refer to Theorem~2.4 in \citet{Gal_Leberling_1977}, but this result is rather theoretical since it depends on the feasible set as one can see from the example in \cref{fig:example_nonessential_objective}.

    There are plenty of approaches to determine nonessential objectives and minimal generating systems, respectively, cf.~\citet{Gal_Leberling_1977,Gal_1980,Malinowska_2006,Thoai_2012}.
    Alternatively, one can use the transformation of the MOLP into a VLP according to \cref{thm:equivalence_molp_vlp} and the above \cref{thm:nonessential_objective_conic_combination}.
    To see how this can be achieved, we first consider a \emph{minimal spanning system} of $\cone(D)$ for a finite set $D \subseteq \R^q$.
    This is a subset $D' \subseteq D$ such that $\cone(D') = \cone(D)$ and for any proper subset $D'' \subsetneq D'$ one has $\cone(D'') \subsetneq \cone(D')$, cf.~\citet{Gal_Leberling_1977,Gal_1980}.
    We now reprove the following statement (cf.~ibid. and~\cite{Thoai_2012}) using our new formulation in \cref{thm:nonessential_objective_conic_combination}, \cref{enm:ne_system_conic_combination}.

    \begin{corollary}
        \label{cor:nonessential_objectives_minimal_spanning_system}
        Consider the problem \molpref, and let $\Pcal$ be the set of its objectives.
        Moreover, let $\Pcal' \subseteq \Pcal$ be a minimal spanning system of $\cone(\Pcal)$.
        Then, $\Pcal \setminus \Pcal'$ is a nonessential system of objectives.
    \end{corollary}
    \begin{proof}
        Regarding \cref{thm:nonessential_objective_conic_combination}, it suffices to show $\Pcal \setminus \Pcal' \subseteq \cone(\Pcal \setminus (\Pcal \setminus \Pcal'))$.
        This follows immediately from $\Pcal \setminus \Pcal' \subseteq \cone(\Pcal)$, the assumption $\cone(\Pcal) = \cone(\Pcal')$, and $\Pcal' = \Pcal \setminus (\Pcal \setminus \Pcal')$.
        \qedhere
    \end{proof}

    \begin{remark}
        \label{rem:extreme_directions_minimal_spanning_system}
        If $\cone(\Pcal)$ is pointed (i.e., it has trivial lineality space~$\{0\}$), the extreme directions $\extdir(\cone(\Pcal))$ of $\cone(\Pcal)$ form a minimal spanning system.
        From the corollary then follows that it suffices to determine those extreme directions and reduce the objective matrix to the rows corresponding to them.
        Note that if $\cone(\Pcal)$ has nontrivial lineality space~$T \neq \{0\}$, one can decompose $\cone(\Pcal)$ via $\cone(\Pcal) = T + (\cone(\Pcal) \cap T^\perp)$~\citep{Rockafellar_1970}, determine the extreme directions of $\cone(\Pcal) \cap T^\perp$, and then add the generating directions of~$T$ in order to obtain a minimal spanning system.
    \end{remark}

    In the transformation of the MOLP into a VLP according to \cref{thm:equivalence_molp_vlp}, the ordering cone~$C$ is given in an H-representation~$\inlineset{z \in \R^k}{L z \geq 0}$, see for example~\citet{Ciripoi_etal_2019} for different representations of polyhedra.
    In practice, if the ordering cone~$C$ has nonempty interior, one can calculate a minimal H-representation of the ordering cone, that is, an H-representation~$C' = \inlineset{z \in \R^k}{L' z \geq 0}$ with a minimal number of inequalities, where $L'$ is a submatrix of~$L$ up to scaling of rows.
    Similar to the proof of \cref{thm:nonessential_objective_conic_combination}, a re-transformation of the VLP yields an equivalent MOLP with objective matrix~$P' = L' R$.
    In fact, if the ordering cone has nonempty interior, the determination of a minimal H-representation is the same as the removal of nonessential objectives due to \cref{cor:nonessential_objectives_minimal_spanning_system}.

    To show this, we denote the transposed rows of $L$ by $\Lcal = \{\ell^1, \ldots, \ell^q\}$.
    Without loss of generality, we assume that the rows of $L$ and $L'$ are both ordered and scaled such that the transposed rows of $L'$ form the set $\Lcal' = \{\ell^1, \ldots, \ell^{\cramped{q'}}\}$.
    Furthermore, we set $\Lcal^* = \Lcal \setminus \Lcal'$.
    Due to the injectivity of $R^\transposed$, the elements of $\Lcal$ and $\Lcal^*$ uniquely correspond to the elements of $\Pcal = R^\transposed[\Lcal]$ and $\Pcal^* = R^\transposed[\Lcal^*]$, respectively.
    By construction, the polar cones of $C$ and $C'$ are equal, i.e., $\cone(\Lcal) = \cone(\Lcal \setminus \Lcal^*)$.
    Thus, $\Lcal^*$ is a subset of $\cone(\Lcal \setminus \Lcal^*)$.
    Applying $R^\transposed$ to this inclusion, one obtains $\Pcal^* \subseteq \cone(\Pcal \setminus \Pcal^*)$, which means that $\Pcal^*$ is nonessential.
    Conversely, if $\Pcal^* = \Pcal \setminus \extdir(\cone(\Pcal))$, the injectivity of $R^\transposed$ yields $\Lcal^* = \Lcal \setminus \extdir(\cone(\Lcal))$.
    As $\extdir(\cone(\Lcal))$ is a minimal spanning system of~$\cone(\Lcal)$ if $\cone(\Pcal)$ is pointed, the polar cone of~$\cone(\extdir(\cone(\Lcal)))$ must be given by a minimal H-representation of $C$.
    If $\cone(\Pcal)$ is not pointed, a similar statement follows using the decomposition of~$\cone(\Pcal)$ mentioned above.

    Therefore, if the used VLP solver inherently calculates a minimal H-representation (or similarly, a V-representation, cf.~\citealt{Ciripoi_etal_2019}), the determination of nonessential objectives is obsolete.
    Otherwise, it is worth a thought executing such a facet or vertex enumeration manually.
    This is because one then searches for nonessential objectives in $\R^k$ instead of $\R^n$, and the ordering cone becomes simpler.
    Alternatively, one can employ the FEC algorithm from~\citet[Section~3]{Thoai_2012}.


    \section{Examples}
    \label{sec:examples}

    In order to verify that the above-described approach indeed provides the desired advantages, we now present two classes of examples.
    We analyze how the computing time depends on either the number~$q$ of objectives or the rank~$k$ of the objective matrix, while holding all the other parameters constant.

    For the construction of the examples, we use the \emph{bensolve tools} package~\citep{Loehne_Weissing_2016,Ciripoi_etal_2019,Ciripoi_etal_bensolvetools} in version~1.3 for GNU~Octave.
    For solving, we apply the \emph{Bensolve} VLP solver~\citep{Loehne_Weissing_Bensolve,Loehne_Weissing_2017} in version~2.1.0 which calculates a solution of the MOLP or VLP, respectively, according to \cref{sec:problem_setting_preliminaries} as well as corresponding feasible points and directions.
    The computer used is an Apple MacBook Air from \num{2023} (\SI{15}{inches} model, Apple M2 chip, \SI{16}{\giga\byte} memory).

    \subsection{A First Series of Examples}
    \label{subsec:example_1}

    The first problem set is constructed from Example~13 of~\citet{Loehne_2023} in~$\R^{12}$.
    Here, we use Octave in version~9.4.0.
    Let $P_0$ be a regular simplex in~$\R^{12}$ with edge length~$1$, symmetrically placed around the origin (generated via the \texttt{simplex} function from \emph{bensolve tools}).
    Recursively, define $P_\ell$ as the Minkowski sum of~$P_{\ell-1}$ and its polar $\inlineset{y \in \R^{12}}{\forall\, x \in P_{\ell-1} : y^\transposed x \leq 1}$.
    From this sequence, we choose~$P_5$.
    We calculate a P-representation~\citep{Ciripoi_etal_2019} of~$P_5$ with \emph{bensolve tools}, i.e., a matrix~$M \in \R^{12 \times n}$ and a polyhedron~$S = \inlineset{x \in \R^n}{a \leq A x \leq b, l \leq x \leq u}$ with~$A \in \R^{m \times n}$ such that $P_5 = \inlineset{M x}{x \in S}$.
    The coefficient matrix~$A$ is of size $m = 1421$ times $n = 2800$.

    Since $P_5$ is full-dimensional, we have $\rk(M) = 12$.
    We now construct a matrix~$T$ supposed to reduce the rank of~$M$ to~$2$ as follows.
    In the first two columns of~$T$, all entries are set to $-\sqrt{3} / 6$, except the entries in rows \numlist{3;6;9;12} of the first column and in rows \numlist{1;4;7;10} of the second column which are set to zero.
    Column $j = 3, \ldots, 12$ is equal to the normalized vector of $(j - 8)$ times the first plus $(j - 7)$ times the second column.
    Finally, we choose the first $q$~rows of~$T M$ as the objective matrix~$P$, where $3 \leq q \leq 12$.
    By construction of~$T$, the rank of~$P$ is~$2$ for each such~$q$.

    We then apply the \emph{Bensolve} VLP solver to solve the resulting MOLPs for every~$q$ via the \texttt{molpsolve} function from \emph{bensolve tools}.
    Moreover, for every test problem, we calculate an LU~decomposition of~$P$ by the \texttt{lu} function from Octave, form the cone~$C$ as a \texttt{polyh} object and apply the VLP solver function \texttt{vlpsolve} from \emph{bensolve tools} to the resulting VLP.
    The computing times of solving each MOLP and VLP (including the decomposition step) are measured.
    Each test problem is run seven times in succession.
    However, only the last five runs are used to calculate the average computing time.

    \Cref{fig:example_1_mean_molp} plots the average computing times against the number of objectives.
    The computing time for the MOLP increases exponentially with~$q$.
    In contrast, the time for solving the VLP, including decomposition, remains nearly constant.

    \begin{figure}[hbt]
        \centering
        \begin{tikzpicture}
    \begin{axis}[ymode = log, basic q, x-q, y-time, legend top left]
        \addplot
        table[
        col sep = comma,
        x index = 0,
        y index = 1]
            {./Data-Fig-5-1-example_1q-5-2.csv};
        \addlegendentry{MOLP}

        \addplot
        table[
        col sep = comma,
        x index = 0,
        y index = 3]
            {./Data-Fig-5-1-example_1q-5-2.csv};
        \addlegendentry{VLP}
    \end{axis}
\end{tikzpicture}
        \caption{
            Computing times (logarithmic scaling) for the MOLP and the VLP (including decomposition) for a number of $q$~objectives for the problem stated in the text. (color online)
        }
        \label{fig:example_1_mean_molp}
    \end{figure}

    \subsection{Randomly Generated Problem Instances}
    \label{subsec:example_2}

    For the second class of examples, we use Octave in version~11.1.0.
    \citet{Rudloff_etal_2017} randomly generated instances of MOLPs in order to compare algorithms.
    The programs are of the form
    \begin{equation}
        \label{eq:numerical_example_molp}
        \min \quad P' x \qq{s.t.} A x \leq b,\ x \geq 0
    \end{equation}
    with objective matrix~$P' \in \R^{k \times n}$, constraint matrix~$A \in \R^{m \times n}$, and right-hand side vector~$b \in \R_+^{m}$.

    In their first set of examples, they independently chose the entries of~$P'$ and~$A$ from a normal distribution with a mean of~\num{0} and a variance of~\num{100}, and the entries of~$b$ from a uniform distribution over~$[0, 10]$.
    We fix parameter values $q$, $k$, $n$, and $m$, and generate \SI{100}{problems} analogously.
    We then form $q - k$ linear combinations of the transposed rows~$p^1, \ldots, p^k$ of~$P'$ via
    \begin{equation}
        \label{eq:example_2_linear_combinations}
        p^i = \sum_{j = 1}^{k} \dfrac{\lambda_{i j}}{\sqrt{k}} p^{\,j} \comma\quad i = k + 1, \ldots, q \comma
    \end{equation}
    and add them together to a new objective matrix~$P = (p^1 \ \cdots \ p^q)^\transposed$.
    The coefficients~$\lambda_{i j}$ are chosen from a standard normal distribution.

    Since it is possible that the first $k$~rows of~$P$ are generated linearly dependently, we calculate $\rk(P)$ before applying the solver.
    If $\rk(P) \neq k$, we skip the problem instance and generate a new one.
    As the employed \emph{Bensolve} version cannot handle non-solid ordering cones, we also calculate the dimension of~$C$ according to \cref{rem:cone_dimension}.
    If it is less than~$k$, we re-generate the problem instance.

    As in \cref{subsec:example_1}, we calculate the computing times for the MOLP and the corresponding VLP, including the LU decomposition.
    For $n = 20$ and $m = 40$, we vary the number~$q$ from \numrange{3}{12} for $k = 2$ and from \numrange{4}{7} for $k = 3$.
    \Cref{fig:example_2q_2_ratio} shows the ratio of the computing times for the VLP and the MOLP for $k = 2$.
    For $q \geq 6$, all ratios are below~\num{1}, meaning the VLP approach yields a lower computing time than the MOLP.
    As the number of objectives increases, the speedup becomes more significant.
    In particular, for $q = 12$, solving the VLP requires less than \SI{1}{\percent} of the solution time for the MOLP in \SI{98}{\percent} of the cases.
    The minimum ratio is below \SI{0.01}{\percent}, and the median is approximately \SI{0.053}{\percent}.
    For $q = 5$ and $q = 4$, \SI{98}{\percent} and \SI{67}{\percent} of cases, respectively, lead to a speedup.
    The VLP approach could solve only \SI{2}{\percent} of the problem instances with $q = 3$ more efficiently than the MOLP approach, with a maximum ratio of about \SI{3.0}{\percent}.
    One reason the VLP approach may not perform better than the MOLP for a small number of objectives is the computation of the decomposition.

    \begin{figure}[hbt]
        \centering
        \begin{tikzpicture}
    \begin{groupplot}
        [boxplots, x-ratio, xmode = log, group style = {rows = 2, xlabels at = edge bottom, vertical sep = 1\baselineskip+1.5ex}, groupplot ylabel = {number $q$ of objectives}]
        \nextgroupplot

        \addplot+[boxplot prepared = {
            draw position = 8,
            lower whisker = 0.006005,
            lower quartile = 0.022450,
            median = 0.033231,
            upper quartile = 0.056335,
            upper whisker = 0.099345
        }] coordinates {
            (0, 0.107873)
            (0, 0.135058)
            (0, 0.124211)
            (0, 0.151390)
            (0, 0.119536)
            (0, 0.218742)
            (0, 0.187285)
        };

        \addplot+[boxplot prepared = {
            draw position = 9,
            lower whisker = 0.002630,
            lower quartile = 0.008745,
            median = 0.013807,
            upper quartile = 0.023186,
            upper whisker = 0.044827
        }] coordinates {
            (0, 0.069263)
            (0, 0.046603)
            (0, 0.086828)
            (0, 0.048839)
        };

        \addplot+[boxplot prepared = {
            draw position = 10,
            lower whisker = 0.000981,
            lower quartile = 0.002896,
            median = 0.004064,
            upper quartile = 0.006511,
            upper whisker = 0.010989
        }] coordinates {
            (0, 0.012773)
            (0, 0.118047)
            (0, 0.027571)
            (0, 0.020445)
            (0, 0.022466)
            (0, 0.013477)
            (0, 0.015638)
            (0, 0.013987)
        };

        \addplot+[boxplot prepared = {
            draw position = 11,
            lower whisker = 0.000322,
            lower quartile = 0.001016,
            median = 0.001735,
            upper quartile = 0.003647,
            upper whisker = 0.006904
        }] coordinates {
            (0, 0.029736)
            (0, 0.018082)
            (0, 0.009689)
            (0, 0.009037)
            (0, 0.008347)
            (0, 0.010901)
            (0, 0.100817)
            (0, 0.023244)
            (0, 0.012061)
            (0, 0.121433)
            (0, 0.011479)
        };

        \addplot+[boxplot prepared = {
            draw position = 12,
            lower whisker = 0.000074,
            lower quartile = 0.000268,
            median = 0.000529,
            upper quartile = 0.001189,
            upper whisker = 0.002421
        }] coordinates {
            (0, 0.118639)
            (0, 0.006807)
            (0, 0.003453)
            (0, 0.010492)
            (0, 0.002855)
            (0, 0.004726)
        };

        \nextgroupplot

        \addplot+[boxplot prepared = {
            draw position = 3,
            lower whisker = 0.972966,
            lower quartile = 1.441441,
            median = 1.601916,
            upper quartile = 1.817467,
            upper whisker = 2.336582
        }] coordinates {
            (0, 2.561265)
            (0, 2.809256)
            (0, 2.556427)
            (0, 3.039269)
            (0, 2.749118)
            (0, 3.526981)
        };

        \addplot+[boxplot prepared = {
            draw position = 4,
            lower whisker = 0.544211,
            lower quartile = 0.752096,
            median = 0.851614,
            upper quartile = 1.050220,
            upper whisker = 1.480328
        }] coordinates {
            (0, 1.730326)
            (0, 1.663142)
            (0, 1.573412)
            (0, 1.693353)
            (0, 1.517181)
            (0, 1.510188)
        };

        \addplot+[boxplot prepared = {
            draw position = 5,
            lower whisker = 0.201774,
            lower quartile = 0.410810,
            median = 0.513114,
            upper quartile = 0.635047,
            upper whisker = 0.953471
        }] coordinates {
            (0, 1.078103)
            (0, 1.203453)
        };

        \addplot+[boxplot prepared = {
            draw position = 6,
            lower whisker = 0.090270,
            lower quartile = 0.153838,
            median = 0.224958,
            upper quartile = 0.292645,
            upper whisker = 0.469306
        }] coordinates {
            (0, 0.538050)
            (0, 0.611514)
            (0, 0.764871)
            (0, 0.913030)
            (0, 0.685241)
            (0, 0.535337)
        };

        \addplot+[boxplot prepared = {
            draw position = 7,
            lower whisker = 0.024089,
            lower quartile = 0.060195,
            median = 0.086386,
            upper quartile = 0.115780,
            upper whisker = 0.189673
        }] coordinates {
            (0, 0.291768)
            (0, 0.299749)
            (0, 0.220109)
            (0, 0.400384)
            (0, 0.396373)
            (0, 0.199967)
            (0, 0.247520)
            (0, 0.244311)
            (0, 0.611151)
        };
    \end{groupplot}
\end{tikzpicture}
        \caption{
            Boxplot showing the ratios of the computing times (logarithmic scaling) for the VLP (including decomposition) and the MOLP, as stated in the text, with an objective matrix of rank~\num{2} containing $q$~rows.
            For each value of~$q$, the sample size is~\num{100}.
            The lines inside the boxes represent the medians, and the boxes indicate the first and third quartiles of the datasets.
            The whiskers show the ratios within the \num{1.5}-interquartile range (i.e., \num{1.5} times the box width) outside the box, and the small vertical lines represent outliers.
        }
        \label{fig:example_2q_2_ratio}
    \end{figure}

    In \cref{fig:example_2q_3_ratio}, the results for $k = 3$ are depicted.
    Except one problem instance in the cases $q = 4$ and $q = 5$, respectively, all problem instances yield a ratio below~\num{1}.
    For $q = 7$, the VLP approach needs between \SIlist{0.014;7.4}{\percent} of the computing time for the MOLP, with a median of \SI{0.14}{\percent}.

    \begin{figure}[hbt]
        \centering
        \begin{tikzpicture}
    \begin{axis}[boxplots, x-ratio, y-q, xmode = log]
        \addplot+[boxplot prepared = {
            draw position = 4,
            lower whisker = 0.161781,
            lower quartile = 0.324338,
            median = 0.398101,
            upper quartile = 0.533611,
            upper whisker = 0.836671
        }] coordinates {
            (0, 1.111767)
        };

        \addplot+[boxplot prepared = {
            draw position = 5,
            lower whisker = 0.012189,
            lower quartile = 0.034138,
            median = 0.053545,
            upper quartile = 0.107803,
            upper whisker = 0.208021
        }] coordinates {
            (0, 0.317281)
            (0, 0.241306)
            (0, 1.701612)
            (0, 0.396610)
            (0, 0.264007)
            (0, 0.323334)
            (0, 0.229957)
        };

        \addplot+[boxplot prepared = {
            draw position = 6,
            lower whisker = 0.000926,
            lower quartile = 0.003577,
            median = 0.007815,
            upper quartile = 0.019055,
            upper whisker = 0.041365
        }] coordinates {
            (0, 0.053478)
            (0, 0.051640)
            (0, 0.053515)
            (0, 0.044324)
        };

        \addplot+[boxplot prepared = {
            draw position = 7,
            lower whisker = 0.000140,
            lower quartile = 0.000663,
            median = 0.001372,
            upper quartile = 0.003230,
            upper whisker = 0.006270
        }] coordinates {
            (0, 0.008160)
            (0, 0.073616)
            (0, 0.014619)
            (0, 0.007114)
            (0, 0.008517)
            (0, 0.008599)
            (0, 0.014280)
            (0, 0.009136)
        };
    \end{axis}
\end{tikzpicture}
        \caption{
            Boxplot showing the ratios of the computing times (logarithmic scaling) for the VLP (including decomposition) and the MOLP, as stated in the text, with an objective matrix of rank~\num{3} containing $q$~rows.
            For each value of~$q$, the sample size is~\num{100}.
            See \cref{fig:example_2q_2_ratio} for notes on how to interpret the chart.
        }
        \label{fig:example_2q_3_ratio}
    \end{figure}

    We also vary the rank~$k$ of the objective matrix from \numrange{2}{4}, holding the number of objectives constant as~$q = 5$.
    Again, $n = 20$ and $m = 40$.
    \Cref{fig:example_2k_ratio} indicates that the closer $q$ and $k$ are, the higher is the speedup.

    \begin{figure}[hbt]
        \centering
        \begin{tikzpicture}
    \begin{axis}[boxplots, x-ratio, y-k, xmode = log]
        \addplot+[boxplot prepared = {
            draw position = 2,
            lower whisker = 0.205331,
            lower quartile = 0.388081,
            median = 0.473466,
            upper quartile = 0.610560,
            upper whisker = 0.942120
        }] coordinates {
            (0, 1.118814)
            (0, 1.016097)
            (0, 1.291991)
            (0, 0.950915)
            (0, 1.055263)
        };

        \addplot+[boxplot prepared = {
            draw position = 3,
            lower whisker = 0.012112,
            lower quartile = 0.031929,
            median = 0.063173,
            upper quartile = 0.121691,
            upper whisker = 0.251273
        }] coordinates {
            (0, 0.416829)
            (0, 0.392586)
            (0, 0.433557)
        };

        \addplot+[boxplot prepared = {
            draw position = 4,
            lower whisker = 0.006983,
            lower quartile = 0.029969,
            median = 0.044580,
            upper quartile = 0.082770,
            upper whisker = 0.156854
        }] coordinates {
            (0, 0.391495)
            (0, 0.213069)
            (0, 0.197265)
            (0, 0.270057)
            (0, 0.921169)
            (0, 0.165757)
            (0, 0.306715)
            (0, 0.175326)
            (0, 0.303949)
            (0, 0.234251)
        };
    \end{axis}
\end{tikzpicture}
        \caption{
            Boxplot showing the ratios of the computing times (logarithmic scaling) for the VLP (including decomposition) and the MOLP, as stated in the text, with an objective matrix of rank~$k$ containing \num{5}~rows.
            For each value of~$k$, the sample size is~\num{100}.
            See \cref{fig:example_2q_2_ratio} for notes on how to interpret the chart.
        }
        \label{fig:example_2k_ratio}
    \end{figure}


    \section{Conclusion}
    \label{sec:conclusion}

    In this article, we examined an approach that reduces the number of objectives in a multi-objective linear program (MOLP) to the rank of the objective matrix.
    A review of the literature reveals approaches for transforming vector linear programs (VLPs) into MOLPs, e.g., \citet{Sawaragi_etal_1985}, and generalizations, e.g., \citet{Dempe_etal_2015}.
    Our proposed method applies the transformation in reverse.

    For the transformation, the objective matrix is decomposed into two factors: one forms a new ordering cone, and the other factor remains as the objective matrix.
    Examples demonstrate the value of this approach, as it can significantly reduce computing time.
    Moreover, it becomes much easier for decision makers to understand and interpret the problem and its solutions.
    This method is not limited to MOLP but can also be applied to VLP.
    More generally, it can be used for arbitrary vector optimization problems with linear objectives that are linearly dependent, independent of the structure of the feasible set.

    This approach can only handle decompositions into matrices whose number of columns or rows, respectively, is minimal, that is, equal to the rank of the objective matrix.
    Using the concept of generalized minimality, we extended the approach to non-minimal decompositions, and stated a more advanced theorem on the equivalence of MOLP and VLP.
    As a corollary, we obtained Theorem~2.11 in~\citet{Gal_Leberling_1977} on nonessential objectives, similar to Corollary~1 in~\citet{Thoai_2012}.
    We illustrated that, after eliminating nonessential objectives, some redundant information (i.e., linear dependence) may remain in the problem.
    This occurs because linear combinations might not be conic combinations.
    With our approach, linear dependence is eliminated.

    \backmatter

    \bmhead{Acknowledgments}

    Special thanks go to Ina Lammel from Fraunhofer IWTM for the suggestions regarding the literature.

    \bmhead{Declaration of Funding}

    This work was supported by the \enquote{Honours Programme for Future Researchers} at Friedrich Schiller University Jena, Germany (project HPJ-UP51).

    \bmhead{Declaration of Competing Interests}

    The authors have no competing interests to declare that are relevant to the content of this article.

    \bibliography{bibliography}

@article{Benson_1998,
    title = {An {{Outer Approximation Algorithm}} for {{Generating All Efficient Extreme Points}} in the {{Outcome Set}} of a {{Multiple Objective Linear Programming Problem}}},
    author = {Benson, Harold P.},
    year = {1998},
    journal = {J Glob Optim},
    volume = {13},
    number = {1},
    pages = {1--24},
    issn = {0925-5001},
    doi = {10.1023/A:1008215702611}
}

@article{Benson_1998a,
    title = {Further {{Analysis}} of an {{Outcome Set-Based Algorithm}} for {{Multiple-Objective Linear Programming}}},
    author = {Benson, Harold P.},
    year = {1998},
    journal = {J Optim Theory Appl},
    volume = {97},
    number = {1},
    pages = {1--10},
    issn = {0022-3239},
    doi = {10.1023/A:1022614814789}
}

@article{Benson_1998b,
    title = {Hybrid {{Approach}} for {{Solving Multiple-Objective Linear Programs}} in {{Outcome Space}}},
    author = {Benson, Harold P.},
    year = {1998},
    journal = {J Optim Theory Appl},
    volume = {98},
    number = {1},
    pages = {17--35},
    issn = {0022-3239},
    doi = {10.1023/A:1022628612489}
}

@incollection{Brockhoff_Zitzler_2006,
    title = {Are {{All Objectives Necessary}}? {{On Dimensionality Reduction}} in {{Evolutionary Multiobjective Optimization}}},
    booktitle = {Parallel Problem Solving from {{Nature}} – PPSN IX: 9th International Conference, {{Reykjavik}}, {{Iceland}}, {{September}} 2006, {{Proceedings}}},
    author = {Brockhoff, Dimo and Zitzler, Eckart},
    editor = {Runarsson, Thomas Philip and Beyer, Hans-Georg and Burke, Edmund and Merelo-Guervós, Juan J. and Whitley, L. Darrell and Yao, Xin},
    year = {2006},
    series = {Lecture notes in computer science},
    number = {4193},
    pages = {533--542},
    publisher = {Springer},
    address = {Berlin, Heidelberg},
    isbn = {3-540-38990-3}
}

@article{Brockhoff_Zitzler_2006a,
    title = {On {{Objective Conflicts}} and {{Objective Reduction}} in {{Multiple Criteria Optimization}}},
    author = {Brockhoff, Dimo and Zitzler, Eckart},
    year = {2006},
    journal = {TIK-Report, ETH Zürich},
    number = {243},
    url = {http://www.cmap.polytechnique.fr/~dimo.brockhoff/publicationListFiles/bz2006a.pdf}
}

@article{Brockhoff_Zitzler_2009,
    title = {Objective {{Reduction}} in {{Evolutionary Multiobjective Optimization}}: {{Theory}} and {{Applications}}},
    author = {Brockhoff, Dimo and Zitzler, Eckart},
    year = {2009},
    journal = {Evol Comput},
    volume = {17},
    number = {2},
    pages = {135--166},
    doi = {10.1162/evco.2009.17.2.135}
}

@article{Cambini_etal_2003,
    title = {Order-{{Preserving Transformations}} and {{Applications}}},
    author = {Cambini, Alberto and Luc, Dinh The and Martein, Laura},
    year = {2003},
    journal = {J Optim Theory Appl},
    volume = {118},
    number = {2},
    pages = {275--293},
    doi = {10.1023/A:1025495204834}
}

@article{Cheung_etal_2016,
    title = {Objective {{Extraction}} for {{Many-Objective Optimization Problems}}: {{Algorithm}} and {{Test Problems}}},
    author = {Cheung, Yiu-ming and Gu, Fangqing and Liu, Hai-Lin},
    year = {2016},
    journal = {IEEE Trans Evol Computat},
    volume = {20},
    number = {5},
    pages = {755--772},
    doi = {10.1109/TEVC.2016.2519758}
}

@article{Ciripoi_etal_2019,
    title = {Calculus of convex polyhedra and polyhedral convex functions by utilizing a multiple objective linear programming solver},
    author = {Ciripoi, Daniel and Löhne, Andreas and Weißing, Benjamin},
    year = {2019},
    journal = {Optimization},
    volume = {68},
    number = {10},
    pages = {2039--2054},
    doi = {10.1080/02331934.2018.1518447}
}

@manual{Ciripoi_etal_bensolvetools,
    title = {bensolve tools: {{Calculus}} of Convex Polyhedra, {{Calculus}} of Polyhedral Convex Functions, Global Optimization, Vector Linear Programming for {{Octave}} and {{Matlab}}},
    author = {Ciripoi, Daniel and Löhne, Andreas and Weißing, Benjamin},
    year = {2024},
    url = {http://tools.bensolve.org},
    version = {1.3}
}

@book{CoelloCoello_etal_2002,
    title = {Evolutionary {{Algorithms}} for Solving Multi-Objective Problems},
    author = {Coello Coello, Carlos A. and Van Veldhuizen, David A. and Lamont, Gary B.},
    year = {2002},
    series = {Genetic algorithms and {{Evolutionary Computation}}},
    number = {5},
    publisher = {Kluwer Academics/Plenum Publishers},
    address = {New York},
    doi = {doi.org/10.1007/978-1-4757-5184-0},
    isbn = {0-306-46762-3}
}

@book{Cohon_1978,
    title = {Multiobjective {{Programming}} and {{Planning}}},
    author = {Cohon, Jared L.},
    year = {1978},
    series = {Mathematics in {{Science}} and {{Engineering}}},
    number = {140},
    publisher = {Academic Press},
    address = {New York},
    isbn = {978-0-12-178350-1}
}

@article{Cohon_etal_1979,
    title = {Generating {{Multiobjective Trade}}‐{{Offs}}: {{An Algorithm}} for {{Bicriterion Problems}}},
    author = {Cohon, Jared L. and Church, Richard L. and Sheer, Daniel P.},
    year = {1979},
    journal = {Water Resources Research},
    volume = {15},
    number = {5},
    pages = {1001--1010},
    issn = {0043-1397},
    doi = {10.1029/WR015i005p01001}
}

@article{Csirmaz_2016,
    title = {Using multiobjective optimization to map the entropy region},
    author = {Csirmaz, László},
    year = {2016},
    journal = {Comput Optim Appl},
    volume = {63},
    number = {1},
    pages = {45--67},
    doi = {10.1007/s10589-015-9760-6}
}

@article{Dauer_1993,
    title = {On degeneracy and collapsing in the construction of the set of objective values in a multiple objective linear program},
    author = {Dauer, Jerald P.},
    year = {1993},
    journal = {Ann Oper Res},
    volume = {46},
    number = {2},
    pages = {279--292},
    issn = {0254-5330},
    doi = {10.1007/BF02023100}
}

@article{Dauer_Liu_1990,
    title = {Solving multiple objective linear programs in objective space},
    author = {Dauer, Jerald P. and Liu, Yi-Hsin},
    year = {1990},
    journal = {Eur J Oper Res},
    volume = {46},
    number = {3},
    pages = {350--357},
    issn = {03772217},
    doi = {10.1016/0377-2217(90)90010-9}
}

@article{Dauer_Saleh_1990,
    title = {Constructing the set of efficient objective values in multiple objective linear programs},
    author = {Dauer, Jerald P. and Saleh, Ossama A.},
    year = {1990},
    journal = {Eur J Oper Res},
    volume = {46},
    number = {3},
    pages = {358--365},
    issn = {03772217},
    doi = {10.1016/0377-2217(90)90011-Y}
}

@article{Dempe_etal_2015,
    title = {On the effects of combining objectives in multi-objective optimization},
    author = {Dempe, Stephan and Eichfelder, Gabriele and Fliege, Jörg},
    year = {2015},
    journal = {Math Meth Oper Res},
    volume = {82},
    number = {1},
    pages = {1--18},
    doi = {10.1007/s00186-015-0501-5}
}

@article{Eckart_Young_1936,
    title = {The {{Approximation}} of {{One Matrix}} by {{Another}} of {{Lower Rank}}},
    author = {Eckart, Carl and Young, Gale},
    year = {1936},
    journal = {Psychometrika},
    volume = {1},
    number = {3},
    pages = {211--218},
    issn = {0033-3123},
    doi = {10.1007/BF02288367}
}

@article{Ehrgott_etal_2012,
    title = {A dual variant of {{Benson}}’s “outer approximation algorithm” for multiple objective linear programming},
    author = {Ehrgott, Matthias and Löhne, Andreas and Shao, Lizhen},
    year = {2012},
    journal = {J Glob Optim},
    volume = {52},
    number = {4},
    pages = {757--778},
    issn = {0925-5001},
    doi = {10.1007/s10898-011-9709-y}
}

@book{Eichfelder_2008,
    title = {Adaptive Scalarization Methods in Multiobjective Optimization},
    author = {Eichfelder, Gabriele},
    editor = {Jahn, Johannes},
    year = {2008},
    series = {Vector {{Optimization}}},
    publisher = {Springer},
    address = {Berlin, Heidelberg},
    doi = {10.1007/978-3-540-79159-1},
    isbn = {978-3-540-79157-7}
}

@article{Engau_Wiecek_2007,
    title = {Cone {{Characterizations}} of {{Approximate Solutions}} in {{Real Vector Optimization}}},
    author = {Engau, Alexander and Wiecek, Margaret M.},
    year = {2007-08-14},
    journal = {J Optim Theory Appl},
    volume = {134},
    number = {3},
    pages = {499--513},
    issn = {0022-3239},
    doi = {10.1007/s10957-007-9235-8}
}

@incollection{Gal_1980,
    title = {A {{Note}} on {{Size Reduction}} of the {{Objective Functions Matrix}} in {{Vector Maximum Problems}}},
    booktitle = {Multiple Criteria Decision Making Theory and {{Application}}: {{Proceedings}} of the Third Conference Hagen/Königswinter, West Germany, {{August}} 20--24, 1979},
    author = {Gal, Tomas},
    editor = {Fandel, Günther and Gal, Tomas},
    year = {1980},
    series = {Lecture {{Notes}} in {{Economics}} and {{Mathematical Systems}}},
    number = {177},
    pages = {74--84},
    publisher = {Springer},
    address = {Berlin, Heidelberg},
    doi = {10.1007/978-3-642-48782-8},
    isbn = {3-540-09963-8}
}

@article{Gal_Leberling_1977,
    title = {Redundant objective functions in linear vector maximum problems and their determination},
    author = {Gal, Tomas and Leberling, Heiner},
    year = {1977},
    journal = {Eur J Oper Res},
    volume = {1},
    number = {3},
    pages = {176--184},
    issn = {03772217},
    doi = {10.1016/0377-2217(77)90025-X}
}

@book{Goepfert_Nehse_1990,
    title = {Vektoroptimierung: Theorie, Verfahren und Anwendungen},
    author = {Göpfert, Alfred and Nehse, Reinhard},
    year = {1990},
    series = {Mathematisch-naturwissenschaftliche Bibliothek},
    number = {74},
    publisher = {BSB Teubner},
    address = {Leipzig},
    isbn = {3-322-00755-3}
}

@book{Golub_VanLoan_2013,
    title = {Matrix {{Computations}}},
    author = {Golub, Gene Howard and Van Loan, Charles F.},
    year = {2013},
    series = {Johns {{Hopkins}} studies in the {{Mathematical Sciences}}},
    edition = {4},
    publisher = {The Johns Hopkins University Press},
    address = {Baltimore},
    isbn = {978-1-4214-0794-4}
}

@article{Hamel_etal_2014,
    title = {Benson type algorithms for linear vector optimization and applications},
    author = {Hamel, Andreas H. and Löhne, Andreas and Rudloff, Birgit},
    year = {2014},
    journal = {J Glob Optim},
    volume = {59},
    number = {4},
    pages = {811--836},
    issn = {0925-5001},
    doi = {10.1007/s10898-013-0098-2}
}

@book{Jahn_2011,
    title = {Vector {{Optimization}}: {{Theory}}, {{Applications}}, and {{Extensions}}},
    author = {Jahn, Johannes},
    year = {2011},
    edition = {2},
    publisher = {Springer},
    address = {Berlin, Heidelberg},
    doi = {10.1007/978-3-642-17005-8},
    isbn = {978-3-642-42330-7}
}

@book{Jolliffe_1986,
    title = {Principal Component Analysis},
    author = {Jolliffe, Ian T.},
    year = {1986},
    series = {Springer {{Series}} in {{Statistics}}},
    publisher = {Springer},
    address = {New York},
    doi = {10.1007/978-1-4757-1904-8},
    isbn = {978-1-4757-1906-2}
}

@article{Lindroth_etal_2010,
    title = {Approximating the {{Pareto}} optimal set using a reduced set of objective functions},
    author = {Lindroth, Peter and Patriksson, Michael and Strömberg, Ann-Brith},
    year = {2010},
    journal = {Eur J Oper Res},
    volume = {207},
    number = {3},
    pages = {1519--1534},
    issn = {03772217},
    doi = {10.1016/j.ejor.2010.07.004}
}

@book{Loehne_2011,
    title = {Vector {{Optimization}} with {{Infimum}} and {{Supremum}}},
    author = {Löhne, Andreas},
    year = {2011},
    series = {Vector {{Optimization}}},
    publisher = {Springer},
    address = {Berlin, Heidelberg},
    doi = {10.1007/978-3-642-18351-5},
    isbn = {978-3-642-18350-8}
}

@article{Loehne_2023,
    title = {Approximate vertex enumeration},
    author = {Löhne, Andreas},
    year = {2023},
    journal = {J Comput Geom},
    volume = {14},
    number = {1},
    pages = {257--286},
    doi = {10.20382/jocg.v14i1a10}
}

@article{Loehne_Weissing_2016,
    title = {Equivalence between polyhedral projection, multiple objective linear programming and vector linear programming},
    author = {Löhne, Andreas and Weißing, Benjamin},
    year = {2016},
    journal = {Math Meth Oper Res},
    volume = {84},
    number = {2},
    pages = {411--426},
    doi = {10.1007/s00186-016-0554-0}
}

@article{Loehne_Weissing_2017,
    title = {The vector linear program solver {{Bensolve}} – notes on theoretical background},
    author = {Löhne, Andreas and Weißing, Benjamin},
    year = {2017},
    journal = {Eur J Oper Res},
    volume = {260},
    number = {3},
    pages = {807--813},
    issn = {03772217},
    doi = {10.1016/j.ejor.2016.02.039}
}

@manual{Loehne_Weissing_Bensolve,
    title = {Bensolve – {{VLP}} solver},
    author = {Löhne, Andreas and Weißing, Benjamin},
    year = {2017},
    url = {http://www.bensolve.org},
    version = {2.1.0}
}

@article{LopezJaimes_etal_2014,
    title = {Objective space partitioning using conflict information for solving many-objective problems},
    author = {López Jaimes, Antonio and Coello Coello, Carlos A. and Aguirre, Hernán and Tanaka, Kiyoshi},
    year = {2014},
    journal = {Inform. Sci.},
    volume = {268},
    pages = {305--327},
    issn = {00200255},
    doi = {10.1016/j.ins.2014.02.002}
}

@article{Malinowska_2006,
    title = {Nonessential objective functions in linear multiobjective optimization problems},
    author = {Malinowska, Agnieszka Barbara},
    year = {2006},
    journal = {Control Cybern.},
    volume = {35},
    number = {4},
    pages = {873--880},
    url = {http://control.ibspan.waw.pl:3000}
}

@article{Mirsky_1960,
    title = {Symmetric gauge functions and unitarily invariant norms},
    author = {Mirsky, Leonid},
    year = {1960},
    journal = {Q J Math},
    volume = {11},
    number = {1},
    pages = {50--59},
    issn = {0033-5606},
    doi = {10.1093/qmath/11.1.50}
}

@article{Pearson_1901,
    title = {On lines and planes of closest fit to systems of points in space},
    author = {Pearson, Karl},
    year = {1901},
    journal = {The London, Edinburgh, and Dublin Philosophical Magazine and Journal of Science},
    volume = {2},
    number = {11},
    pages = {559--572},
    issn = {1941-5982},
    doi = {10.1080/14786440109462720}
}

@article{Perez-Gallardo_etal_2018,
    title = {Combining {{Multi-Objective Optimization}}, {{Principal Component Analysis}} and {{Multiple Criteria Decision Making}} for ecodesign of photovoltaic grid-connected systems},
    author = {Perez-Gallardo, J. Raul and Azzaro-Pantel, Catherine and Astier, Stéphan},
    year = {2018},
    journal = {Sustainable Energy Technol. Assess.},
    volume = {27},
    pages = {94--101},
    issn = {22131388},
    doi = {10.1016/j.seta.2018.03.008}
}

@article{Pozo_etal_2012,
    title = {On the use of {{Principal Component Analysis}} for reducing the number of environmental objectives in multi-objective optimization: {{Application}} to the design of chemical supply chains},
    author = {Pozo, Carlos and Ruiz-Femenia, Rubén and Caballero, J. A. and Guillén-Gosálbez, Gonzalo and Jiménez, Laureano},
    year = {2012},
    journal = {Chem Eng Sci},
    volume = {69},
    number = {1},
    pages = {146--158},
    issn = {00092509},
    doi = {10.1016/j.ces.2011.10.018}
}

@book{Rockafellar_1970,
    title = {Convex {{Analysis}}},
    author = {Rockafellar, R. Tyrrell},
    year = {1970},
    series = {Princeton mathematical {{Series}}},
    number = {28},
    publisher = {Princeton University Press},
    address = {Princeton},
    isbn = {978-0-691-08069-7}
}

@incollection{Rosenthal_2023,
    title = {Selection {{Strategies}} for a {{Balanced Multi-}} or {{Many-Objective Molecular Optimization}} and {{Genetic Diversity}}: {{A Comparative Study}}},
    booktitle = {Evolutionary Multi-Criterion Optimization},
    author = {Rosenthal, Susanne},
    editor = {Emmerich, Michael and Deutz, André and Wang, Hao and Kononova, Anna V. and Naujoks, Boris and Li, Ke and Miettinen, Kaisa and Yevseyeva, Iryna},
    year = {2023},
    volume = {13970},
    pages = {490--503},
    publisher = {Springer Nature Switzerland},
    address = {Cham},
    doi = {10.1007/978-3-031-27250-9_35},
    isbn = {978-3-031-27250-9}
}

@article{Rudloff_etal_2017,
    title = {A parametric simplex algorithm for linear vector optimization problems},
    author = {Rudloff, Birgit and Ulus, Firdevs and Vanderbei, Robert},
    year = {2017},
    journal = {Math Program},
    volume = {163},
    number = {1--2},
    pages = {213--242},
    issn = {0025-5610},
    doi = {10.1007/s10107-016-1061-z}
}

@book{Sawaragi_etal_1985,
    title = {Theory of Multiobjective Optimization},
    author = {Sawaragi, Yoshikazu and Nakayama, Hirotaka and Tanino, Tetsuzo},
    year = {1985},
    series = {Mathematics in {{Science}} and {{Engineering}}},
    number = {176},
    publisher = {Academic Press},
    address = {Orlando},
    isbn = {0-12-620370-9}
}

@article{Saxena_etal_2013,
    title = {Objective {{Reduction}} in {{Many-Objective Optimization}}: {{Linear}} and {{Nonlinear Algorithms}}},
    author = {Saxena, Dhish Kumar and Duro, João A. and Tiwari, Ashutosh and Deb, Kalyanmoy and Zhang, Qingfu},
    year = {2013},
    journal = {IEEE Trans. Evol. Computat.},
    volume = {17},
    number = {1},
    pages = {77--99},
    doi = {10.1109/TEVC.2012.2185847}
}

@article{Shao_Ehrgott_2008,
    title = {Approximately solving multiobjective linear programmes in objective space and an application in radiotherapy treatment planning},
    author = {Shao, Lizhen and Ehrgott, Matthias},
    year = {2008},
    journal = {Math Meth Oper Res},
    volume = {68},
    number = {2},
    pages = {257--276},
    issn = {1432-2994},
    doi = {10.1007/s00186-008-0220-2}
}

@article{Shao_Ehrgott_2008a,
    title = {Approximating the nondominated set of an {{MOLP}} by approximately solving its dual problem},
    author = {Shao, Lizhen and Ehrgott, Matthias},
    year = {2008},
    journal = {Math Meth Oper Res},
    volume = {68},
    number = {3},
    pages = {469--492},
    issn = {1432-2994},
    doi = {10.1007/s00186-007-0194-5}
}

@article{Thoai_2012,
    title = {Criteria and dimension reduction of linear multiple criteria optimization problems},
    author = {Thoai, Nguyen V.},
    year = {2012},
    journal = {J Glob Optim},
    volume = {52},
    number = {3},
    pages = {499--508},
    doi = {10.1007/s10898-011-9764-4}
}

@article{Weidner_1990,
    title = {Complete efficiency and interdependencies between objective functions in vector optimization},
    author = {Weidner, Petra},
    year = {1990},
    journal = {Z. Oper. Res.},
    volume = {34},
    number = {2},
    pages = {91--115},
    doi = {10.1007/BF01415973}
}

@book{Yu_1985,
    title = {Multiple-Criteria Decision Making: {{Concepts}}, {{Techniques}}, and {{Extensions}}},
    author = {Yu, Po-Lung},
    year = {1985},
    series = {Mathematical {{Concepts}} and {{Methods}} in {{Science}} and {{Engineering}}},
    number = {30},
    publisher = {Springer},
    address = {New York},
    doi = {10.1007/978-1-4684-8395-6},
    isbn = {978-0-306-41965-2}
}

\end{document}